\documentclass[10pt,reqno]{amsart}

\usepackage{enumerate}
\usepackage{amsmath, amssymb, amsthm,amsfonts}
\usepackage{mathrsfs}
\usepackage{esint}
\usepackage{xcolor}
\usepackage{mathtools}
\usepackage{hyperref}
\usepackage{bm}
\usepackage{accents}
\usepackage[foot]{amsaddr}
\usepackage{todonotes}

\makeatletter

\numberwithin{equation}{section}

\newtheorem{thm}{Theorem}[section]
\newtheorem{theorem}[thm]{Theorem}
\newtheorem{lemma}[thm]{Lemma}

\newtheorem{prop}[thm]{Proposition}
\theoremstyle{definition}

\newtheorem{definition}[thm]{Definition}

\newtheorem{assumption}[thm]{Assumption}
\theoremstyle{remark}

\newtheorem{remark}[thm]{\bf{Remark}}

\newcommand\aint{-\hspace{-0.38cm}\int}

\newcommand\bH{\mathbb{H}}

\newcommand\bL{\mathbb{L}}
\newcommand\bM{\mathbb{M}}
\newcommand\bN{\mathbb{N}}

\newcommand\bR{\mathbb{R}}

\newcommand\cA{\mathcal{A}}
\newcommand\cB{\mathcal{B}}
\newcommand\cC{\mathcal{C}}
\newcommand\cD{\mathcal{D}}

\newcommand\cH{\mathcal{H}}

\newcommand\cL{\mathcal{L}}

\newcommand\cP{\mathcal{P}}

\newcommand\sL{\mathscr{L}}

\newcommand{\mysection}[1]{\section{#1}}

\begin{document}

\title[Degenerate linear equations]{Sobolev estimates for parabolic and elliptic equations in divergence form with degenerate coefficients}

\author{Hongjie Dong$^{1}$}
\address{$^1$ Division of Applied Mathematics, Brown University, 182 George Street, Providence, RI 02912, USA}
\email{Hongjie\_Dong@brown.edu}
\thanks{H. Dong was partially supported by the NSF under agreement DMS2350129.}

\author{Junhee Ryu$^{2}$}
\address{$^2$ School of Mathematics, Korea Institute for Advanced Study, 85 Hoegi-ro, Dongdaemun-gu, Seoul, 02455, Republic of Korea}
\email{junhryu@kias.re.kr}
\thanks{J. Ryu was supported by a KIAS Individual Grant (MG101501) at Korea Institute for Advanced Study}

\subjclass[2020]{35J70, 35K65, 35D30, 35R05}

\keywords{Degenerate linear equations, divergence form, existence and uniqueness, weighted Sobolev spaces}

\begin{abstract}
We study a class of degenerate parabolic and elliptic equations in divergence form in the upper half space $\{x_d>0\}$. The leading coefficients are of the form $x_d^2a_{ij}$, where $a_{ij}$ are bounded, uniformly elliptic, and measurable in $(t,x_d)$ except $a_{dd}$, which is measurable in $t$ or $x_d$. Additionally, they have small bounded mean oscillations in the other spatial variables. We obtain the well-posedness and regularity of solutions in weighted mixed-norm Sobolev spaces.
\end{abstract}

\maketitle

\section{Introduction}

In this paper, we study the existence, uniqueness, and regularity of solutions in weighted mixed-norm Sobolev spaces to a class of parabolic and elliptic equations in the upper half space. The leading coefficients are the product of $x_d^2$ and uniformly nondegenerate bounded measurable matrix-valued functions, which are degenerate when $x_d\to 0^+$ and singular when $x_d\to \infty$. 

Throughout this paper, we always assume that a matrix of coefficients $(a_{ij})$ is measurable and satisfies the following ellipticity condition and boundedness condition
\begin{equation} \label{ellip}
  \nu|\xi|^2\leq a_{ij}\xi_i\xi_j, \quad |a_{ij}|\leq \nu^{-1},
\end{equation}
and $a_0,(b_i),(\hat{b}_i), c$, and $c_0$ are measurable and satisfy
\begin{equation}
						\label{eq1.2}
  |b_i|,|\hat{b}_i|,|c| \leq K
\end{equation}
and
\begin{equation} \label{eq7151434}
  K^{-1}\leq a_0,c_0 \leq K.
\end{equation}
Let $\sL_p$ be the second-order linear
parabolic operator with degenerate coefficients defined by
\begin{equation*}
  \sL_p u:= a_0 u_t - x_d^2 D_i(a_{ij}D_{j}u) + x_d b_iD_iu + x_dD_i(\hat{b}_iu) + cu.
\end{equation*}

For $T\in(-\infty,\infty]$, $d\in\bN$, and $\bR^{d}_+ := \{(x_1,\dots,x_d)\in\bR^d : x_d>0\}$, we investigate degenerate parabolic equations in divergence form:
\begin{equation} \label{maindiv}
  \sL_p u +\lambda c_0 u =D_iF_i+f
\end{equation}
in $\Omega_T := (-\infty,T)\times\bR^d_+$, as well as the corresponding Cauchy problems in $\Omega_{0,T}:=(0,T)\times \bR^d_+$. Here, $\lambda\geq0$, and $F$ and $f$ are given measurable forcing terms.
We also consider the following elliptic equations:
\begin{equation} \label{ellmaindiv}
  \sL_e u +\lambda c_0 u =D_iF_i + f
\end{equation}
in $\bR^d_+$,  where $\sL_e$ is the elliptic operator defined by
\begin{equation*}
  \sL_e u:= -x_d^2 D_i(a_{ij}D_{j}u) + x_d b_iD_iu + x_dD_i(\hat{b}_iu) +cu.
\end{equation*}
We note that in the first-order terms, $x_d$ is placed in front of $D_i$. Thus, the equations \eqref{maindiv} and \eqref{ellmaindiv} are invariant under the scaling $(t,x)\to (t,sx)$ and $x\to sx$, respectively, for any $s>0$.

 Equations \eqref{maindiv} and \eqref{ellmaindiv} appear in various problems. When $d=1$, a notable example of the parabolic equation is the Black-Scholes-Merton equation
\begin{equation} \label{eq5090018}
    u_t + \frac{1}{2}\sigma^2x^2D_{xx}u + rxD_xu -ru=0 \, \text{ in } \, \Omega_{0,T},
\end{equation}
where $\sigma,r>0$. Here, we emphasize that in the literature, the equation is typically considered in the reverse time direction with a terminal condition $u(T,x)=h(x)$ in place of an initial one.
Regarding \eqref{ellmaindiv}, if $d=1$, then the equation is the Euler equation, one of the well-known ordinary differential equations,
\begin{equation} \label{eq5091531}
    -ax^2D_{xx}u + xbD_xu + cu =f \, \text{ in } \, \bR_+,
\end{equation}
where $a,b,c \in \bR$ are constants.
For multi-dimensional case, the equations are useful and important in various problems (see, for instance, \cite{HX24,V89} and the references therein). For example, in the introduction of \cite{HX24}, it is noted that \eqref{ellmaindiv} appears in the linearization of the Loewner-Nirenberg problem
\begin{equation*}
    \Delta u=\frac{d(d-2)}{4}u^{\frac{d+2}{d-2}} \, \text{ in } \, \bR^d,
\end{equation*}
which is nonlinear and degenerate.
One further motivation to study our equations comes from degenerate viscous Hamilton-Jacobi equations, where a model equation is given by
\begin{equation*}
    u_t(t,x)-x_d^{\alpha}\Delta u(t,x) +\lambda u(t,x) + H(t,x,D_xu)=0 \text{ in } \Omega_T.
\end{equation*}
Here, $\alpha>0$, and $H:\Omega_T\times \bR^d\to\bR$ is a given smooth Hamiltonian. When $H=0$ and $\alpha=2$, this equation is a special case of \eqref{maindiv}. We refer the reader to \cite{DPT24} for more information.

The solution spaces of interest are the weighted Sobolev spaces $H_{p,\theta}^1$, defined as
\begin{equation*}
  H_{p,\theta}^1:=\{u: u,x_dD_xu \in L_{p}(x_d^{\theta-1}dx)\}.
\end{equation*}
For the parabolic equation \eqref{maindiv}, we present weighted mixed-norm spaces; we consider the Muckenhoupt class of weights in time. See Section \ref{sec_function} for the definitions of these spaces. Such spaces were introduced in \cite[Section 2.6.3]{ML68} for $p=2$ and $\theta=1$, and they were generalized in a unified manner for $p\in(1,\infty)$, $\theta\in\bR$, and fractional derivatives in \cite{Khalf}. The necessity of these weighted spaces came from the theory of stochastic partial differential equations (SPDEs). See, for instance, \cite{K94,KLline,KLspace}. After the work of \cite{Khalf}, second-order nondegenerate equations have been extensively studied in the weighted Sobolev spaces. We refer the reader to \cite{KK2004,KL13,KN09,DK15,S24}.

It is worth noting that no boundary condition is imposed in this paper.  While boundary conditions can be addressed with additional information about the forcing terms or the structure of the equations (see, for example, \cite[Remark 4.7]{V89}), we focus on solving the equations under minimal assumptions.
Additionally, boundary data of functions in $H_{p,\theta}^1$ can be analyzed by imposing further constraints. For instance, in \cite{D08}, it was shown that when $\theta\in(-p,0)$, $u\in H_{p,\theta}^1$, and $x_du\in L_p(x_d^{\theta-1}dx)$, the trace of $u$ exists and is zero. We also refer the reader to \cite{DP20,DP23JFA}, where the authors studied the elliptic equations with the prototype
\begin{equation*}
  x_d^2\Delta u+\alpha x_d D_du-\lambda x_d^2u=f
\end{equation*}
with certain boundary conditions. Due to the presence of the term $\lambda x_d^2u$ instead of $\lambda u$ in these papers, boundary conditions were imposed, and the elliptic problems were solved in different function spaces.

The purpose of this paper is to obtain maximal regularity results for solutions in $H_{p,\theta}^1$. For example, for the elliptic problem \eqref{ellmaindiv}, we prove
\begin{equation} \label{eq11221300}
  \int_{\bR^d_+} \left(|(1+\sqrt{\lambda})u|^p + |x_dD_xu|^p\right)x_d^{\theta-1} dx \leq N \int_{\bR^d_+} \left(|x_d^{-1}F|^p + \left|\frac{f}{1+\sqrt{\lambda}}\right|^p\right) x_d^{\theta-1} dx,
\end{equation}
under two distinct cases for $\theta$ and $\lambda$. First, for any $\theta\in\bR$, we require the condition $\lambda\geq\lambda_0$ where $\lambda_0\geq0$ is sufficiently large. Second, when $\lambda=0$, the range of $\theta$ becomes restricted. In this case, the lower-order coefficients $(b_i),(\hat{b}_i)$, and $c$ are ``effective'' in the sense that the range of $\theta$ depends also on the following ratios of coefficients:
\begin{equation*}
  \frac{b_d}{a_{dd}}=n_b, \quad \frac{\hat{b}_{d}}{a_{dd}}=n_{\hat{b}}, \quad \frac{c}{a_{dd}}=n_c.
\end{equation*}
We note that the range of $\theta$ is optimal in the sense that it is a necessary and sufficient condition for the solvability. We will demonstrate this in our future work, where the corresponding non-divergence equations are studied.
We also remark that the two zeroth-order terms, $cu$ and $\lambda c_0 u$, are introduced and play distinct roles in the analysis.

In this paper, the leading coefficients $a_{ij}$ are assumed to be measurable in $(t,x_d)$ except $a_{dd}$, which is assumed to be measurable in $t$ or $x_d$. Additionally, they have small bounded mean oscillations (BMO) in the remaining spatial directions. 
This setup is motivated by the classical (nondegenerate) heat equation, where this class of the leading coefficients is optimal in the sense that the unique $L_p$ solvability fails if $a_{dd}$ is measurable in both $(t,x_d)$. See \cite{K16} for a counterexample and \cite{D20,DK18} and the references therein for the solvability of nondegenerate equations.

Let us give the main ideas and organization of this paper. In Section \ref{sec_main}, we introduce the functions spaces, assumptions and our main results.
In Section \ref{sec_apriori}, we provide the a priori estimates in unmixed-norm spaces when the coefficients are simple. The notions of simple coefficients are introduced in Assumptions \ref{ass_simple1} and \ref{ass_simple2}. The proof is divided into zeroth-order and higher-order estimates. For zeroth-order estimates, we test the equation with a suitable test function and use weighted Hardy's inequality. In the case when the coefficients are measurable in $x_d$, a crucial step of our proof is to transform the equation appropriately. For higher-order estimates, we apply a localization argument from \cite{Khalf}.
In Section \ref{sec_simple}, we first prove the existence of solutions in $L_p((-\infty,T);H_{p,\theta}^1)$. Then we use a level set argument to derive weighted mixed-norm estimates for simple coefficients.
The proofs for the parabolic equations are given in Section \ref{sec_bmo}. The main approach begins by handling compactly supported solutions using an embedding and an interpolation argument. A partition of unity with weight is then applied to obtain the desired estimates.
The elliptic case is addressed in Section \ref{sec_ell}. Finally, in Section \ref{sec_ink}, we prove a ``crawling of ink spots'' lemma for our setting.

In the literature, there are numerous results on degenerate elliptic and parabolic equations.
In particular, when $d=1$, the solutions are well-known for special cases. For instance, for \eqref{eq5090018}, we have
\begin{equation*}
    u(t,x)=e^{-r(T-t)}\int_0^\infty h(y)p(t,x,y)dy,
\end{equation*}
where the transition density 
\begin{equation*}
    p(t,x,y):=\frac{1}{\sigma y\sqrt{2\pi(T-t)}}\exp{\left\{-\frac{1}{2(T-t)\sigma^2}\left[\log\left(\frac{y}{x}\right)-(r-\frac{1}{2}\sigma^2)(T-t)\right]^2\right\}}.
\end{equation*}
(see e.g. \cite[Section 16.6]{S04}). In the elliptic case, if the quadratic polynomial $az^2+(b+1)z-c=0$
has two distinct real roots $\alpha$ and $\beta$, then the general solution of \eqref{eq5091531} is
\begin{equation} \label{eq5091832}
  Ax^{-\alpha} + Bx^{-\beta},
\end{equation}
where $A,B\in\bR$. Using this, one can find an explicit representation formula of the solution (see \eqref{eq10101418}).

Although there are many results for more general cases and higher dimensional equations, we only give a review on regularity results.
 We first focus on $L_p$ regularity results relevant to \eqref{maindiv} and \eqref{ellmaindiv}.
The elliptic problem \eqref{ellmaindiv} with constant coefficients was considered in \cite{Khalf,MNS22,MNS24}, using different approaches: elementary analysis and probabilistic representation in \cite{Khalf}, and a semigroup approach in \cite{MNS22,MNS24}.
In \cite{K07}, the author studied the Cauchy problem \eqref{maindiv} where the leading coefficients are uniformly continuous. See also \cite{K08} for the corresponding result on SPDEs.
For equations with higher-order degeneracy in the form $x_d^\alpha\Delta$ with $\alpha\geq2$, we refer the reader to \cite{FMP12,V89}. In particular, for the case $\alpha=2$, the estimate \eqref{eq11221300} with $\theta=1$ was proved in \cite{FMP12}.
Compared to \cite{FMP12,K07,Klec,MNS22,V89}, we consider  a substantially larger class of coefficients in weighted mixed-norm spaces. 

Next, we describe H\"older regularity results for the following degenerate elliptic problem defined on bounded domains:
\begin{equation*}
  \rho^{2s}a_{ij}D_{ij}u+\rho^{s}b_iD_iu+cu=f,
\end{equation*}
where $s\geq1$ and $\rho$ is a regularized distance function. In \cite{V89}, the author studied this equation in weighted H\"older spaces. In \cite{GL91,HX24}, higher-order weighted regularity of solutions was obtained when $s=1$. We also remark that in \cite{V89, GL91}, the zero Dirichlet boundary condition was considered by imposing the additional assumption that $f$ vanishes on the boundary.

We also provide a brief review on equations involving operators with lower-order degeneracy, such as $x_d^\alpha\Delta$ with $\alpha<2$. In \cite{DPT23, DPT24}, the authors studied equations with the prototype
\begin{equation*}
  u_t-x_d^\alpha\Delta u +\lambda u=f
\end{equation*}
 under the zero Dirichlet boundary condition in weighted Sobolev spaces. We note that our operators can be viewed as the limiting case as $\alpha\to2$, but our results cannot be obtained by (formally) taking the limit in their results. There has also been extensive research on equations with first-order degeneracy, i.e., $\alpha=1$. In particular, the following equation
 \begin{equation*}
   u_t-x_d\Delta u -\beta D_du +\lambda u=f, \quad \lambda\geq0, \,\beta>0,
 \end{equation*}
 appears in the study of porous media equations and parabolic Heston equations.
  We refer the reader to \cite{FMPP07,Koch} for weighted Sobolev estimates and to \cite{DH98,FP13} for weighted H\"older estimates, respectively.

Lastly, a different class of parabolic equations with singular-degenerate coefficients was studied in a series of papers \cite{DP20,DP21TAMS, DP23,DP23JFA}. The authors showed the wellposedness and regularity estimates in weighted Sobolev spaces. In these papers, the weights of coefficients of $u_t$ and $D_x^2u$ appear in a balanced way, which plays a crucial role in the analysis and functional space settings.

We finish the introduction by summarizing the notation used in this paper.
 We use ``$:=$'' or ``$=:$'' to denote a definition.
For non-negative functions $f$ and $g$, we write $f\approx g$ if there exists a constant $N>0$ such that $N^{-1}f\leq g\leq Nf$.
  By $\bN$, we denote the natural number system. We denote $\bN_0:=\bN\cup\{0\}$. As usual, $\bR^d$ stands for the Euclidean space of points $x=(x_1,\dots,x_d)=(x',x_d)$. We also denote $B_r(x):=\{y\in\bR^d : |x-y|<r\}$ and write $\bR:=\bR^1$.
We use $D^n_x u$ to denote the partial derivatives of order $n\in\bN_0$ with respect to the space variables, and $D_xu:=D_x^1u$. We also denote
\begin{equation*}
  D_iu=\frac{\partial u}{\partial x_i}, \quad D_{ij}u=\frac{\partial^2 u}{\partial x_i \partial x_j}.
\end{equation*}

\section{Main results} \label{sec_main}

\subsection{Function spaces} \label{sec_function}

We first introduce function spaces which will be used in this paper. We denote by $L_{p,\theta}=L_{p,\theta}(\bR^d_+)$ the set of all measurable functions $u$ defined on $\bR^d_+$ satisfying
\begin{equation*}
  \|u\|_{L_{p,\theta}}:=\left( \int_{\bR^d_+} |u|^p x_d^{\theta-1} dx \right)^{1/p}<\infty.
\end{equation*}
Denote
\begin{equation*}
  H_{p,\theta}^1:=\{u: u, x_dD_xu \in L_{p,\theta}\}.
\end{equation*}
Here, the norm in $H_{p,\theta}^1$ is given by
\begin{equation} \label{eq4251017}
  \|u\|_{H_{p,\theta}^1} := \left( \sum_{i=0}^n \int_{\bR^d_+} \left( |u|^p + |x_dD_xu|^p\right) x_d^{\theta-1} dx \right)^{1/p}.
\end{equation}
By \cite[Corollary 2.3]{Khalf}, the norm \eqref{eq4251017} is equivalent to
\begin{equation} \label{equivnorm}
  \|u\|_{H_{p,\theta}^1} \approx \left(\sum_{m=-\infty}^\infty e^{m(\theta+d-1)} \|u(e^m\cdot)\zeta\|_{W_p^1(\bR^d)}^p\right)^{1/p},
\end{equation}
where $W_p^1(\bR^d)$ is the Sobolev space of order $1$ and $\zeta\in C_c^\infty(\bR_+)$ is a non-negative function such that
\begin{equation} \label{eq716908}
  \sum_{n=-\infty}^\infty \zeta^p(e^{x_d-n})\geq1.
\end{equation}
We also remark that $H_{p,\theta}^1$ in this paper is denoted by $H_{p,\theta+d-1}^1$ in \cite{Khalf}.

For $p,q\geq1$, a given weight $\omega=\omega(t)$ on $(-\infty,T)$, and functions defined on $\Omega_T$, we define
 \begin{equation*}
   \bL_{q,p,\theta,\omega}(T):=L_q((-\infty,T),\omega dt;L_{p,\theta}), \quad \bH_{q,p,\theta,\omega}^1(T):=L_q((-\infty,T), \omega dt;H_{p,\theta}^1).
 \end{equation*}
In the case when $p=q$ and $\omega=1$, we write $\bL_{p,\theta}(T):=\bL_{p,p,\theta,1}(T)$ and $\bH_{p,\theta}^1(T):=\bH_{p,p,\theta,1}^1(T)$. We also denote $\bL_{q,p,\theta}(T):=\bL_{q,p,\theta,1}(T)$, $\bH_{q,p,\theta}^1(T):=\bH_{q,p,\theta,1}^1(T)$, $\bL_{p,\theta,\omega}(T):=\bL_{p,p,\theta,\omega}(T)$ and $\bH_{p,\theta,\omega}^1(T):=\bH_{p,p,\theta,\omega}^1(T)$.
Following the argument in \cite[Theorem 1.19, Remark 5.5]{Khalf}, it can be seen that $C_c^\infty(\bR^d_+)$ and $C_c^\infty((-\infty,T)\times\bR^d_+)$ are dense in $H_{p,\theta}^1$ and $\bH_{q,p,\theta,\omega}^1(T)$, respectively. For $-\infty\leq S<T\leq\infty$, and functions defined on $\Omega_{S,T}:=(S,T)\times\bR^d_+$, we denote $\bL_{q,p,\theta,\omega}(S,T)$ and $\bH^n_{q,p,\theta,\omega}(S,T)$ in a similar way.

We introduce the function space for $u_t=\partial_t u$ as follows. For a fixed function $a_0$, we define
\begin{align*}
  \widetilde{\bH}_{q,p,\theta,\omega}^{-1}(T):=\{u:a_0u=D_iF_i+f, \text{ where } x_d^{-1}F,f\in \bL_{q,p,\theta,\omega}(T) \},
\end{align*}
that is equipped with the norm
\begin{align*}
  \|u\|_{\widetilde{\bH}_{q,p,\theta,\omega}^{-1}(T)}:=\inf\{ \|x_d^{-1}F\|_{\bL_{q,p,\theta,\omega}(T)} + \|f\|_{\bL_{q,p,\theta,\omega}(T)}: a_0u=D_iF_i+f \}.
\end{align*}
We also define $\widetilde{\bH}_{q,p,\theta,\omega}^{-1}(S,T)$ in a similar way.
Now we define the solution space $\cH_{q,p,\theta,\omega}^1(T)$ to be the closure of $C_c^\infty((-\infty,T]\times \bR^d_+)$ under the norm
\begin{equation*}
  \|u\|_{\cH_{q,p,\theta,\omega}^1(T)}:=\|u\|_{\bH_{q,p,\theta,\omega}^1(T)}+\|u_t\|_{\widetilde{\bH}_{q,p,\theta,\omega}^{-1}(T)}.
\end{equation*}
Here, for $u_t\in \widetilde{\bH}_{q,p,\theta,\omega}^{-1}(T)$, the equality $a_0u_t=D_iF_i+f$ is understood in the weak formulation; for any $\varphi \in C_c^\infty((-\infty,T)\times\bR^d_+)$,
\begin{equation*} 
  -\int_{\Omega_T} a_0 u \varphi_t dxdt = -\int_{\Omega_T} F_i D_i\varphi dxdt + \int_{\Omega_T} f \varphi dxdt
\end{equation*}
if $a_0=a_0(x_d)$, and
\begin{equation*}
  -\int_{\Omega_T} u \varphi_t dxdt = -\int_{\Omega_T} \frac{1}{a_0} F_i D_i\varphi dxdt + \int_{\Omega_T} \frac{1}{a_0} f \varphi dxdt
\end{equation*}
if $a_0=a_0(t)$.
We also write $\cH_{p,\theta}^1(T) := \cH_{p,p,\theta,1}^1(T)$ and $\cH_{p,\theta,\omega}^1(T) := \cH_{p,p,\theta,\omega}^1(T)$.
For equations defined on $(S,T)\times\bR^d_+$, we consider the solution space $\mathring{\cH}_{q,p,\theta,\omega}^1(S,T)$, which is defined as the closure of the set of functions $u\in C_c^\infty([S,T]\times \bR^d_+)$ with $u(S,\cdot)=0$, equipped with the norm
\begin{equation*}
  \|u\|_{\mathring{\cH}_{q,p,\theta,\omega}^1(S,T)}:=\|u\|_{\bH_{q,p,\theta,\omega}^1(S,T)}+\|u_t\|_{\widetilde{\bH}_{q,p,\theta,\omega}^{-1}(S,T)}.
\end{equation*}

We recall the definition of the $A_p$ Muckenhoupt class of weights.

\begin{definition}
  Let $p\in(1,\infty)$. A locally integrable function $\omega:\bR \to (0,\infty)$ is said to be in the $A_p(\bR)$ Muckenhoupt class of weights if
  \begin{equation*}
    [\omega]_{A_p(\bR)}:=[\omega]_{A_p}:=\sup_{r>0, t\in\bR} \left( \aint_{t-r}^{t} \omega(s)ds \right)\left( \aint_{t-r}^{t} \omega(s)^{1/(1-p)} ds \right)^{p-1} <\infty.
  \end{equation*}
\end{definition}

\subsection{Parabolic equations}

In this subsection, we present our main results regarding parabolic equations.

We first impose the following regularity assumptions on the coefficients, where the parameters $\rho_0\in(1/2,1)$ and $\gamma_0>0$ will be specified later.

\begin{assumption}[$\rho_0,\gamma_0$] \label{ass_lead}
For each $x_0\in \bR^d_+$ and $\rho\in(0,\rho_0x_{0d}]$, there exist coefficients $[a_{ij}]_{\rho,x_0}$ and $[c_{0}]_{\rho,x_0}$ satisfying \eqref{ellip}-\eqref{eq7151434}. Moreover,
\begin{itemize}
  \item $a_0, [a_{dd}]_{\rho, x_0}$, and $[c_0]_{\rho, x_0}$ either:
  \begin{itemize}
    \item depend only on $x_d$ for all $\rho\in(0,\rho_0x_{0d}]$, or 
    \item depend only on $t$ for all $\rho\in(0,\rho_0x_{0d}]$,
  \end{itemize}

  \item $[a_{ij}]_{\rho,x_0}$ depend only on $(t,x_d)$ for $(i,j)\neq(d,d)$,

  \item for any $t\in(-\infty,T)$ and $\rho\in(0,\rho_0 x_{0d}]$,
   \begin{align} \label{eq11081151}
    &\aint_{B_\rho(x_0)} \Big(\left| a_{ij}(t,y)- [a_{ij}]_{\rho,x_0}(t,y_d) \right| + \left| c_0(t,y)- [c_0]_{\rho,x_0}(t,y_d) \right|\Big) \, dy < \gamma_0.
  \end{align}
\end{itemize}
\end{assumption}

To handle the case when $\lambda=0$, we impose the following stronger assumption on the coefficients.

\begin{assumption}[$\rho_0,\gamma_0$] \label{ass_coeff}
For each $x_0\in \bR^d_+$ and $\rho\in(0,\rho_0x_{0d}]$, there exist coefficients $[a_{ij}]_{\rho,x_0}, [b_{i}]_{\rho,x_0}, [\hat{b}_{i}]_{\rho,x_0}, [c]_{\rho,x_0}$, and $[c_0]_{\rho,x_0}$ satisfying \eqref{ellip}-\eqref{eq7151434}, and the ratio condition
    \begin{eqnarray*}
    \frac{[b_{d}]_{\rho,x_0}}{[a_{dd}]_{\rho,x_0}}=n_b, \quad \frac{[\hat{b}_{d}]_{\rho,x_0}}{[a_{dd}]_{\rho,x_0}}=n_{\hat{b}}, \quad \frac{[c]_{\rho,x_0}}{[a_{dd}]_{\rho,x_0}}=n_{c}
  \end{eqnarray*}
  for some $n_b,n_{\hat{b}},n_c\in\bR$ independent of $x_0$ and $\rho$.
  Moreover,
\begin{itemize}
  \item one of the following is satisfied:
  \begin{itemize}
    \item $[a_{dd}]_{\rho, x_0}, [b_d]_{\rho, x_0}, [\hat{b}_d]_{\rho, x_0}$, and  $[c]_{\rho, x_0}$ are constant, and $a_0$ and $[c_0]_{\rho, x_0}$ depend only on $x_d$ for all $\rho\in(0,\rho_0x_{0d}]$,
    \item $a_0, [a_{dd}]_{\rho, x_0}, [b_d]_{\rho, x_0}, [\hat{b}_d]_{\rho, x_0}, [c]_{\rho, x_0}$, and $[c_0]_{\rho, x_0}$ depend only on $t$ for all $\rho\in(0,\rho_0x_{0d}]$,
  \end{itemize}

  \item $[a_{ij}]_{\rho,x_0}$ depend only on $(t,x_d)$ for $(i,j)\neq (d,d)$,

  \item $[b_{i}]_{\rho,x_0}$ and $[\hat{b}_{i}]_{\rho,x_0}$ depend only on $(t,x_d)$ for $i\neq d$,

  \item for any $t\in(-\infty,T)$,
  \begin{align} \label{eq6030053}
  &\aint_{B_\rho(x_0)} \Big( \left| a_{ij}(t,y)- [a_{ij}]_{\rho,x_0}(t,y_d) \right| + \left| b_{i}(t,y)- [b_{i}]_{\rho,x_0}(t,y_d) \right| \nonumber
    \\
    &\qquad\qquad+ \left| \hat{b}_{i}(t,y)- [\hat{b}_{i}]_{\rho,x_0}(t,y_d) \right| + \left| c(t,y)- [c]_{\rho,x_0}(t,y_d) \right|  \nonumber
        \\
    &\qquad\qquad+ \left| c_0(t,y)- [c_0]_{\rho,x_0}(t,y_d) \right| \Big) \, dy < \gamma_0.
  \end{align}
\end{itemize}
\end{assumption}

Here, we remark that in \eqref{eq11081151} and \eqref{eq6030053}, coefficients are appropriately understood based on the variables of dependency. For instance, $[a_{dd}]_{\rho,x_0}(t,y_d)$ can be either $[a_{dd}]_{\rho,x_0}(y_d)$ or $[a_{dd}]_{\rho,x_0}(t)$.

\begin{remark} \label{rem5101542}
Under Assumption \ref{ass_lead} $(\rho_0,\gamma_0)$ or Assumption \ref{ass_coeff} $(\rho_0,\gamma_0)$, for any $\theta\in \bR$, there is a constant $N=N(d,\rho_0,\theta)$ such that
  \begin{equation} \label{eq4022235}
    \sup_{(t,x)\in\bR^{d+1}_+} \sup_{\rho\in(0,\rho x_d]} \aint_{x_d-\rho}^{x_d+\rho} \aint_{B'_\rho(x')} \left| a_{ij}(t,y',y_d)- [a_{ij}]_{\rho,x_0}(t,y_d) \right| dy' \mu_\theta(dy_d)<N\gamma_0,
  \end{equation}
  where $\mu_\theta(dy_d) := y_d^{\theta-1}dy_d$ and $B'_\rho(x')=\{y'\in \bR^{d-1}:|x'-y'|<\rho\}$.
This can be shown by using $\mu_{\theta} ((x_d-\rho,x_d+\rho))\approx x_d^{\theta}$ and $y_d^\theta\approx x_d^\theta$ for $\rho\in(0,\rho_0 x_d]$ and $y_d\in (x_d-\rho,x_d+\rho)$.

Moreover, one can also obtain \eqref{eq4022235} for the lower-order coefficients.
\end{remark}

Now we present the definition of weak solutions to \eqref{maindiv}.

\begin{definition} \label{weak}
  Let $p,q\in (1,\infty)$, $\theta\in\bR$, $T\in(-\infty,\infty]$, $\omega\in A_q(\bR)$, and $F,f\in L_{1,loc}(\Omega_T)$. 

  $(i)$ In the case when $a_0=a_0(x_d)$, we say that $u\in \bH^1_{q,p,\theta,\omega}(T)$ is a weak solution to \eqref{maindiv} if
  \begin{align*}
    &-\int_{\Omega_T} a_0 u\varphi_t dxdt + \int_{\Omega_T} a_{ij} D_juD_{i}(x_d^2\varphi) dxdt + \int_{\Omega_T} x_d \varphi b_i D_iudxdt 
    \\
    &\quad- \int_{\Omega_T} \hat{b}_i uD_i(x_d\varphi) dxdt + \int_{\Omega_T} c u\varphi dxdt + \int_{\Omega_T} \lambda c_0 u\varphi dxdt 
    \\
    &= - \int_{\Omega_T} F_i D_i\varphi dxdt + \int_{\Omega_T} f\varphi dxdt
  \end{align*}
  for any $\varphi \in C_c^\infty((-\infty,T)\times \bR^d_+)$.

  $(ii)$ In the case when $a_0=a_0(t)$, we say that $u\in \bH^1_{q,p,\theta,\omega}(T)$ is a weak solution to \eqref{maindiv} if
  \begin{align*}
    &-\int_{\Omega_T} u\varphi_t dxdt + \int_{\Omega_T} \frac{a_{ij}}{a_0} D_juD_{i}(x_d^2\varphi) dxdt + \int_{\Omega_T} x_d \varphi \frac{b_i}{a_0} D_iu dxdt 
    \\
    &\quad- \int_{\Omega_T} \frac{\hat{b}_i}{a_0} uD_i(x_d\varphi) dxdt + \int_{\Omega_T} \frac{c}{a_0} u\varphi dxdt + \int_{\Omega_T} \lambda \frac{c_0}{a_0} u\varphi dxdt 
    \\
    &= - \int_{\Omega_T} \frac{1}{a_0} F_i D_i\varphi dxdt + \int_{\Omega_T} \frac{1}{a_0} f\varphi dxdt
  \end{align*}
  for any $\varphi \in C_c^\infty((-\infty,T)\times \bR^d_+)$.
\end{definition}

Below is our first result for \eqref{maindiv}.
Note that we always assume that the coefficients $a_0,(b_i),(\hat{b}_i), c$, and $c_0$ satisfy \eqref{ellip}-\eqref{eq7151434}.

\begin{theorem} \label{thm_para_div}
  Let $T\in(-\infty,\infty]$, $p,q\in(1,\infty)$, $\theta\in\bR$, and $K_0$ be a constant such that $[\omega]_{A_q}\leq K_0$. 

$(i)$ There exist
 \begin{eqnarray*}
   \rho_0=\rho_0(d,p,q,\theta,\nu,K,K_0)\in(1/2,1)
 \end{eqnarray*}
 sufficiently close to $1$, a sufficiently small number
 \begin{eqnarray*}
   \gamma_0=\gamma_0(d,p,q,\theta,\nu,K,K_0)>0,
 \end{eqnarray*}
 and a sufficiently large number
  \begin{eqnarray*}
   \lambda_0=\lambda_0(d,p,q,\theta,\nu,K,K_0)\geq 0,
 \end{eqnarray*}
 such that if Assumption \ref{ass_lead} $(\rho_0,\gamma_0)$ is satisfied, then the following assertions hold true.
  For any $\lambda\geq \lambda_0$, and $x_d^{-1}F:=x_d^{-1}(F_1,\dots,F_d), f \in \bL_{q,p,\theta,\omega}(T)$, there is a unique solution $u\in \cH_{q,p,\theta,\omega}^1(T)$ to \eqref{maindiv}.
Moreover, for this solution, we have
\begin{align} \label{estdiv}
  &(1+\sqrt{\lambda}) \|u\|_{\bL_{q,p,\theta,\omega}(T)} + \|x_dD_xu\|_{\bL_{q,p,\theta,\omega}(T)} \nonumber
  \\
  &\leq N \left( \|x_d^{-1}F\|_{\bL_{q,p,\theta,\omega}(T)} + \frac{1}{1+\sqrt{\lambda}}\|f\|_{\bL_{q,p,\theta,\omega}(T)} \right),
\end{align}
where $N=N(d,p,q,\theta,\nu,K,K_0)$.

$(ii)$ Let $n_b,n_{\hat{b}},n_c\in\bR$, and assume that the quadratic equation 
\begin{equation} \label{eq5091641}
  z^2+(1+n_b+n_{\hat{b}})z-n_c=0
\end{equation}
has two distinct real roots $\alpha<\beta$. Then there exist
 \begin{eqnarray*}
   \rho_0=\rho_0(d,p,q,\theta,n_b,n_{\hat{b}},n_c,\nu,K,K_0)\in(1/2,1)
 \end{eqnarray*}
 sufficiently close to $1$, and a sufficiently small number
 \begin{eqnarray*}
   \gamma_0=\gamma_0(d,p,q,\theta,n_b,n_{\hat{b}},n_c,\nu,K,K_0)>0
 \end{eqnarray*}
 such that under Assumption \ref{ass_coeff} $(\rho_0,\gamma_0)$, the assertions in $(i)$ hold with $\lambda_0=0$ and $\alpha p<\theta<\beta p$, where the dependencies of the constant $N$ are replaced by $d,p,q,\theta,n_b,n_{\hat{b}},n_c,\nu,K$, and $K_0$.
\end{theorem}

\begin{remark}
In Theorem \ref{thm_para_div} $(i)$, $(b_i),(\hat{b}_i)$, and $c$ are assumed to be merely bounded measurable, and $\theta\in\bR$ is arbitrary, although $\lambda_0\geq0$ is sufficiently large.
On the other hand, in $(ii)$ of the theorem, the conditions on the lower-order coefficients and $\theta$ are restricted, but we can take $\lambda=0$.
\end{remark}

We also consider the Cauchy problems on a finite time interval:
\begin{equation} \label{div_finite}
  a_0 u_t - x_d^2 D_i(a_{ij}D_{j}u) + x_d b_iD_iu + x_dD_i(\hat{b}_iu) + cu +\lambda c_0 u =D_iF+f
\end{equation}
in $\Omega_{0,T}:=(0,T)\times\bR^d_+$ with $u(0,\cdot)=0$.

Below we present the definition of weak solutions to \eqref{div_finite} with general initial data, although we only deal with zero initial values.

\begin{definition}
  Let $p\in (1,\infty)$, $\theta\in\bR$, $T\in(0,\infty]$, $\omega\in A_p(\bR)$, $u_0\in L_{1,loc}(\bR^d_+)$, and $F,f\in L_{1,loc}(\Omega_{0,T})$. 

$(i)$ In the case when $a_0=a_0(x_d)$, we say that $u\in \bH^1_{q,p,\theta,\omega}(0,T)$ is a weak solution to \eqref{div_finite} with $u(0,\cdot)=u_0$ if
  \begin{align*}
    &-\int_{\bR^d_+} a_0 u_0\varphi(0,\cdot) dx - \int_{\Omega_{0,T}} a_0 u\varphi_t dxdt
    \\
    &\quad+ \int_{\Omega_{0,T}} a_{ij}D_iuD_{j}(x_d^2\varphi) dxdt + \int_{\Omega_{0,T}} x_d b_i \varphi D_iudxdt
    \\
    &\quad- \int_{\Omega_{0,T}} \hat{b}_i uD_i(x_d\varphi) dxdt + \int_{\Omega_{0,T}} cu\varphi dxdt + \int_{\Omega_{0,T}} \lambda c_0 u\varphi dxdt 
    \\
    &= - \int_{\Omega_{0,T}} F_i D_i\varphi dxdt + \int_{\Omega_{0,T}} f\varphi dxdt
  \end{align*}
  for any $\varphi \in C_c^\infty([0,T)\times \bR^d_+)$.

  $(ii)$ In the case when $a_0=a_0(t)$, we say that $u\in \bH^1_{q,p,\theta,\omega}(0,T)$ is a weak solution to \eqref{div_finite} with $u(0,\cdot)=u_0$ if
    \begin{align*}
    &-\int_{\bR^d_+} u_0\varphi(0,\cdot) dx - \int_{\Omega_{0,T}} u\varphi_t dxdt 
    \\
    &\quad+ \int_{\Omega_{0,T}} \frac{a_{ij}}{a_0} D_iu D_{j}(x_d^2\varphi) dxdt + \int_{\Omega_{0,T}} x_d \frac{b_i}{a_0} \varphi D_iudxdt
    \\
    &\quad- \int_{\Omega_{0,T}} \frac{\hat{b}_i}{a_0} uD_i(x_d\varphi) dxdt + \int_{\Omega_{0,T}} \frac{c}{a_0} u\varphi dxdt + \int_{\Omega_{0,T}} \lambda \frac{c_0}{a_0} u\varphi dxdt 
    \\
    &= - \int_{\Omega_{0,T}} \frac{1}{a_0} F_i D_i\varphi dxdt + \int_{\Omega_{0,T}} \frac{1}{a_0} f \varphi dxdt
  \end{align*}
  for any $\varphi \in C_c^\infty([0,T)\times \bR^d_+)$.
\end{definition}

Our next result is regarding the solvability of the Cauchy problem. 

\begin{theorem} \label{thm_finite_div}
  Let $T\in(0,\infty)$, $p,q\in(1,\infty)$, $\theta\in\bR$, and $K_0$ be a constant such that $[\omega]_{A_q}\leq K_0$.
 Then there exist
 \begin{eqnarray*}
   \rho_0=\rho_0(d,p,q,\theta,\nu,K,K_0)\in(1/2,1)
 \end{eqnarray*}
 sufficiently close to $1$, and a sufficiently small number
 \begin{eqnarray*}
   \gamma_0=\gamma_0(d,p,q,\theta,\nu,K,K_0)>0,
 \end{eqnarray*}
 such that if Assumption \ref{ass_lead} $(\rho_0,\gamma_0)$ is satisfied, then the following assertions hold.
 For any $\lambda\geq 0$, and $x_d^{-1}F:=x_d^{-1}(F_1,\dots,F_d),f \in \bL_{q,p,\theta,\omega}(T)$, there is a unique solution $u\in \mathring{\cH}_{q,p,\theta,\omega}^1(0,T)$ to \eqref{div_finite} with zero initial condition $u(0,\cdot)=0$.
Moreover, for this solution, we have
\begin{align} \label{eq10021627}
  &(1+\sqrt{\lambda}) \|u\|_{\bL_{q,p,\theta,\omega}(0,T)} + \|x_dD_xu\|_{\bL_{q,p,\theta,\omega}(0,T)} \nonumber
  \\
  &\leq N \left( \|x_d^{-1}F\|_{\bL_{q,p,\theta,\omega}(0,T)} + \frac{1}{1+\sqrt{\lambda}}\|f\|_{\bL_{q,p,\theta,\omega}(0,T)} \right),
\end{align}
where $N=N(d,p,q,\theta,\nu,K,K_0,T)$.
\end{theorem}

\begin{remark}
In Theorem \ref{thm_finite_div}, by letting the constant $N$ depend also on $T$, we can consider arbitrary $\lambda\geq0$ and $\theta\in\bR$.
\end{remark}

\subsection{Elliptic equations}

In this subsection, we state our main results for elliptic equations.

We state our regularity assumptions on the coefficients, where the parameters $\rho_0\in(1/2,1)$ and $\gamma_0>0$ will be specified later.

\begin{assumption}[$\rho_0,\gamma_0$] \label{ass_ellip_lead}
For each $x_0\in \bR^d_+$ and $\rho\in(0,\rho_0x_{0d}]$, there exist coefficients $[a_{ij}]_{\rho,x_0}$ and $[c_0]_{\rho,x_0}$ satisfying \eqref{ellip}-\eqref{eq7151434}.
  Moreover,
\begin{itemize}
    \item $[a_{ij}]_{\rho, x_0}$ and $[c_0]_{\rho, x_0}$ depend only on $x_d$,

  \item we have
  \begin{align*}
    &\aint_{B_\rho(x_0)} \Big( \left| a_{ij}(y)- [a_{ij}]_{\rho,x_0}(y_d) \right| + \left| c_0(y)- [c_0]_{\rho,x_0}(y_d) \right| \Big) \, dy < \gamma_0.
  \end{align*}
\end{itemize}
\end{assumption}

To handle the case when $\lambda=0$, we impose the following stronger assumption on the coefficients.

\begin{assumption}[$\rho_0,\gamma_0$] \label{ass_ellip}
For each $x_0\in \bR^d_+$ and $\rho\in(0,\rho_0x_{0d}]$, there exist coefficients $[a_{ij}]_{\rho,x_0}, [b_{i}]_{\rho,x_0}, [\hat{b}_{i}]_{\rho,x_0}, [c]_{\rho,x_0}$, and $[c_0]_{\rho,x_0}$ satisfying \eqref{ellip}-\eqref{eq7151434}, and the ratio condition
    \begin{eqnarray*}
    \frac{[b_{d}]_{\rho,x_0}}{[a_{dd}]_{\rho,x_0}}=n_b, \quad \frac{[\hat{b}_{d}]_{\rho,x_0}}{[a_{dd}]_{\rho,x_0}}=n_{\hat{b}}, \quad \frac{[c]_{\rho,x_0}}{[a_{dd}]_{\rho,x_0}}=n_{c}
  \end{eqnarray*}
  for some $n_b,n_{\hat{b}},n_c\in\bR$ independent of $x_0$ and $\rho$.
  Moreover,
\begin{itemize}
    \item $[a_{dd}]_{\rho, x_0}, [b_d]_{\rho, x_0}, [\hat{b}_d]_{\rho, x_0}$, and  $[c]_{\rho, x_0}$ are constant,

  \item $[a_{ij}]_{\rho,x_0}$ depend only on $x_d$ for $(i,j)\neq (d,d)$,

  \item $[b_{i}]_{\rho,x_0}$ and $[\hat{b}_{i}]_{\rho,x_0}$ depend only on $x_d$ for $i\neq d$,

  \item we have
  \begin{align*}
    &\aint_{B_\rho(x_0)} \Big( \left| a_{ij}(y)- [a_{ij}]_{\rho,x_0}(y_d) \right| + \left| b_{i}(y)- [b_{i}]_{\rho,x_0}(y_d) \right| \nonumber
    \\
    &\qquad\qquad+\left| \hat{b}_{i}(y)- [\hat{b}_{i}]_{\rho,x_0}(y_d) \right| + \left| c(y)- [c]_{\rho,x_0}(y_d) \right| \nonumber
        \\
    &\qquad\qquad+ \left| c_0(y)- [c_0]_{\rho,x_0}(y_d) \right| \Big) \, dy < \gamma_0.
  \end{align*}
\end{itemize}
\end{assumption}

Now we present the definition of weak solutions to \eqref{ellmaindiv}.

\begin{definition}
  Let $p\in (1,\infty)$, $\theta\in\bR$, and $F,f\in L_{1,loc}(\bR^d_+)$. 
We say that $u\in H^1_{p,\theta}$ is a weak solution to \eqref{ellmaindiv} if
  \begin{align*}
    &\int_{\bR^d_+} a_{ij} D_juD_{i}(x_d^2\varphi) dx + \int_{\bR^d_+} x_d b_i \varphi D_iu dx 
    \\
    &\quad- \int_{\bR^d_+} \hat{b}_i uD_i(x_d\varphi) dx + \int_{\bR^d_+} c u\varphi dx + \int_{\bR^d_+} \lambda c_0 u\varphi dx 
    \\
    &= - \int_{\bR^d_+} F_i D_i\varphi dx + \int_{\bR^d_+} f\varphi dx
  \end{align*}
  for any $\varphi \in C_c^\infty(\bR^d_+)$.
\end{definition}

The following is our results for the equation \eqref{ellmaindiv}.

\begin{theorem} \label{thm_ell}
Let $p\in(1,\infty)$, and $\theta\in\bR$.  

$(i)$ There exist
 \begin{eqnarray*}
   \rho_0=\rho_0(d,p,\theta,\nu,K)\in(1/2,1)
 \end{eqnarray*}
 sufficiently close to $1$, a sufficiently small number
 \begin{eqnarray*}
   \gamma_0=\gamma_0(d,p,\theta,\nu,K)>0,
 \end{eqnarray*}
 and a sufficiently large number
  \begin{eqnarray*}
   \lambda_0=\lambda_0(d,p,\theta,\nu,K)\geq 0,
 \end{eqnarray*}
 such that under Assumption \ref{ass_ellip_lead} $(\rho_0,\gamma_0)$, the following hold.
For any $\lambda\geq \lambda_0$, and $x_d^{-1}F:=x_d^{-1}(F_1,\dots,F_d),f \in L_{p,\theta}$, there is a unique solution $u\in H_{p,\theta}^1$ to \eqref{ellmaindiv}.
Moreover, for this solution, we have
\begin{align} \label{eq10081100}
  &(1+\sqrt{\lambda}) \|u\|_{L_{p,\theta}} + \|x_dD_xu\|_{L_{p,\theta}} \leq N \left( \|x_d^{-1}F\|_{L_{p,\theta}} + \frac{1}{1+\sqrt{\lambda}}\|f\|_{L_{p,\theta}} \right),
\end{align}
where $N=N(d,p,\theta,\nu,K)$.

$(ii)$ Let $n_b,n_{\hat{b}},n_c\in\bR$, and assume that the quadratic equation 
\begin{equation*}
  z^2+(1+n_b+n_{\hat{b}})z-n_c=0
\end{equation*}
has two distinct real roots $\alpha<\beta$. Then there exist
 \begin{eqnarray*}
   \rho_0=\rho_0(d,p,\theta,n_b,n_{\hat{b}},n_c,\nu,K)\in(1/2,1)
 \end{eqnarray*}
 sufficiently close to $1$, and a sufficiently small number
 \begin{eqnarray*}
   \gamma_0=\gamma_0(d,p,\theta,n_b,n_{\hat{b}},n_c,\nu,K)>0
 \end{eqnarray*}
such that under Assumption \ref{ass_ellip} $(\rho_0,\gamma_0)$, the assertions in $(i)$ hold with $\lambda_0=0$ and $\alpha p<\theta<\beta p$, where the dependencies of the constant $N$ are replaced by $d,p,\theta,n_b,n_{\hat{b}},n_c,\nu$, and $K$.

$(iii)$ When $d=1$ and $\lambda=0$, the assertions in $(ii)$ hold true for $\theta\in \bR\setminus\{\alpha p,\beta p\}$.
\end{theorem}

\begin{remark}
    In Theorems \ref{thm_para_div} and \ref{thm_ell}, it is necessary for our analysis that \eqref{eq5091641} has two distinct real roots. In particular, we require the range of $\theta$ in order to prove zeroth order estimate in Lemma \ref{lem_zero_div}. One can also observe from \eqref{eq5091832} that since \eqref{eq5091641} has two distinct real roots, behavior of solution to the elliptic problem can be described as a linear combination of power functions.
\end{remark}

\mysection{A priori estimates} \label{sec_apriori}

In this section, we obtain a priori estimates for $u\in C_c^\infty((-\infty,T]\times \bR^d_+)$.
First, we state weighted Hardy's inequality.

\begin{lemma} \label{lem_Hardy}
  Let $p\in(1,\infty)$, $\theta\in\bR$, and $u\in C_c^\infty(\bR^d_+)$. Then we have
  \begin{equation*}
    \frac{|\theta|^2}{p^2}\int_{\bR^d_+} |u|^p x_d^{\theta-1} dx \leq \int_{\bR^d_+} |u|^{p-2} (D_{d}u)^2 x_d^{\theta+1} dx.
  \end{equation*}
\end{lemma}

\begin{proof}
  By one-dimensional Hardy's inequality (see e.g. Theorem 5.2 in the preface of \cite{K85}), for any $\theta\neq0$,
\begin{equation*}
  \int_0^\infty |v(r)|^2 r^{\theta-1} dr \leq \frac{4}{|\theta|^2} \int_0^\infty |v'(r)|^2 r^{\theta+1} dr.
\end{equation*}
We rewrite this inequality to include the case $\theta=0$ as follows;
\begin{equation*}
  \frac{|\theta|^2}{4}\int_0^\infty |v(r)|^2 r^{\theta-1} dr \leq \int_0^\infty |v'(r)|^2 r^{\theta+1} dr.
\end{equation*}
Then it remains to put $v(r)=|u(x',r)|^{p/2}$ and integrate both sides with respect to $x'$. The lemma is proved.
\end{proof}

We introduce two classes of simple coefficients.
 The following two assumptions correspond to the classes of $[\cdot]_{\rho,x_0}$-coefficients introduced in Assumptions \ref{ass_lead} $(\rho_0,\gamma_0)$ and \ref{ass_coeff} $(\rho_0,\gamma_0)$, respectively.

\begin{assumption} \label{ass_simple1}
\textbf{} 
   \begin{itemize}
    \item $a_0, a_{dd}$, and $c_0$ depend only on the same single variable, either $x_d$ or $t$,

  \item $a_{ij}$ depend only on $(t,x_d)$ for $(i,j)\neq (d,d)$,

  \item $b_i=\hat{b}_i=c=0$ for all $i$.
\end{itemize}
\end{assumption}

\begin{assumption} \label{ass_simple2}
\textbf{} 
   \begin{itemize}
    \item one of the following is satisfied:
    \begin{itemize}
      \item $a_{dd}, b_d, \hat{b}_d$, and $c$ are constant, and $a_0$ and $c_0$ depend only on $x_d$,

      \item $a_0, a_{dd}, b_d, \hat{b}_d, c$, and $c_0$ depend only on $t$,
    \end{itemize}

  \item $a_{ij}$ depend only on $(t,x_d)$ for $(i,j)\neq (d,d)$,

  \item $b_i$ and $\hat{b}_i$ depend only on $(t,x_d)$ for $i\neq d$,

  \item we have
      \begin{equation} \label{ratio}
    \frac{b_d}{a_{dd}}=n_b, \quad \frac{\hat{b}_{d}}{a_{dd}}=n_{\hat{b}}, \quad \frac{c}{a_{dd}}=n_c
  \end{equation}
  for some $n_b,n_{\hat{b}}, n_c\in \bR$.
\end{itemize}
\end{assumption}

Now we estimate zeroth order regularity of solutions to equations with simple coefficients.

\begin{lemma} \label{lem_zero_div}
Let $T\in(-\infty,\infty]$, $\theta\in\bR$, $\lambda\geq0$, and $p\in[2,\infty)$. Assume that $u\in C_c^\infty((-\infty,T]\times\bR^d_+)$ satisfies 
\begin{equation} \label{zero_div}
\sL_p u +\lambda c_0u = a_0 u_t - x_d^2 D_i(a_{ij}D_{j}u) + x_d b_iD_iu + x_dD_i(\hat{b}_iu) + cu +\lambda c_0 u = D_iF_i+f
\end{equation}
in $\Omega_T$, where $x_d^{-1}F,f\in \bL_{p,\theta}(T)$.

$(i)$ There exists a sufficiently large $\lambda_0=\lambda_0(d,p,\theta,\nu,K)\geq0$
such that under Assumption \ref{ass_simple1}, the following holds true.
 For any $\lambda\geq\lambda_0$, we have
\begin{align} \label{eq4111550}
  &\sup_{t\leq T} \int_{\bR^d_+} |u|^p(t,\cdot) x_d^{\theta-1} dx + \int_{\Omega_T} |u|^p x_d^{\theta-1} dxdt \nonumber
  \\
  &\leq N \int_{\Omega_T} \left(|x_d^{-1}F|^p + |u|^{p-1}|f|\right) x_d^{\theta-1} dxdt,
\end{align}
where $N=N(d,p,\theta,\nu,K)$.

$(ii)$ Let $n_b,n_{\hat{b}},n_c\in \bR$, and assume that the quadratic equation 
\begin{equation} \label{quad_div}
 z^2+(1+n_b+n_{\hat{b}})z-n_c=0
\end{equation}
has two distinct real roots $\alpha <\beta$.
Then under Assumption \ref{ass_simple2}, \eqref{eq4111550} holds for any $\lambda\geq0$ and $\alpha p<\theta <\beta p$, where the constant $N$ depends on $d,p,\theta,n_b,n_{\hat{b}},n_c,\nu$, and $K$.
\end{lemma}

\begin{proof}
First, we prove $(ii)$.

We claim that it suffices to consider the case when
\begin{equation} \label{eq5100023}
  \theta+1+n_b-(p-1)n_{\hat{b}}=0.
\end{equation}
 Indeed, let $\gamma\in\bR$ and denote $v:=x_d^\gamma u$. Then $v$ satisfies
\begin{align} \label{eq4131657}
  &v_t-x_d^2 D_i(a_{ij}D_{j}v) + x_d(\gamma a_{di}+b_i)D_iv + x_dD_i((\gamma a_{id} + \hat{b}_i) v) \nonumber
  \\
  &\quad + (c-\gamma(\gamma+1)a_{dd}-\gamma b_d-\gamma \hat{b}_d)v + \lambda c_0 v \nonumber
  \\
  &= D_i(x_d^\gamma F_i)-\gamma x_d^{\gamma-1}F_d + x_d^\gamma f.
\end{align}
Here, the quadratic polynomial corresponding to equation \eqref{eq4131657} is
\begin{equation*}
  z^2 + (1+n_b+n_{\hat{b}}+2\gamma)z -(n_c-\gamma(\gamma+1)-\gamma n_b-\gamma n_{\hat{b}})=0,
\end{equation*}
and its two (distinct) roots are $\alpha-\gamma$ and $\beta-\gamma$. Thus, it suffices to consider the modified equation \eqref{eq4131657} for $(\alpha-\gamma)p<\theta-\gamma p<(\beta-\gamma)p$.
Notice that for \eqref{eq4131657}, the corresponding value of \eqref{eq5100023} is $(2-2p)\gamma+\theta+1+n_b-(p-1)n_{\hat{b}}$.
Then we choose $\gamma\in\bR$ so that $(2-2p)\gamma+\theta+1+n_b-(p-1)n_{\hat{b}}=0$, which proves the claim.

Now, assume that \eqref{eq5100023} holds. We first consider the case $a_{0}=a_{0}(x_d)$.
  Let us test \eqref{zero_div} by $|u|^{p-2}ux_d^{\theta-1}$ in $\Omega_s$ where $s\leq T$. Then
  \begin{align} \label{eq4132035}
    &\int_{\Omega_s} a_0u_t |u|^{p-2}u x_d^{\theta-1} dxdt + \int_{\Omega_s} a_{ij} D_i(x_d^{\theta+1} |u|^{p-2}u) D_{j}u dxdt \nonumber
    \\
    &+ \int_{\Omega_s} b_i|u|^{p-2}u D_{i} u x_d^{\theta} dxdt - \int_{\Omega_T} \hat{b}_i u D_{i}(|u|^{p-2}u x_d^{\theta}) dxdt + \int_{\Omega_s} c |u|^p x_d^{\theta-1} dxdt \nonumber
    \\
    &+ \lambda \int_{\Omega_s} c_0 |u|^p x_d^{\theta-1} dxdt = - \int_{\Omega_s} D_i (x_d^{\theta-1} |u|^{p-2}u) F_i dxdt + \int_{\Omega_s} |u|^{p-2}uf x_d^{\theta-1} dxdt.
  \end{align}
By integration by parts,
\begin{align} \label{eq6281608}
  \int_{\Omega_s} a_0u_t |u|^{p-2}u x_d^{\theta-1} dxds &= \int_{\bR_+^d} \int_{-\infty}^s \frac{a_0}{p} (|u|^p)_t x_d^{\theta-1} dxdt \nonumber
  \\
  &= \frac{1}{p}\int_{\bR_+^d} a_0|u|^p(s,x) x_d^{\theta-1} dx.
\end{align}
For the second term on the left-hand side of \eqref{eq4132035},
\begin{align} \label{eq5101401}
  &\int_{\Omega_s} a_{ij} D_i(x_d^{\theta+1} |u|^{p-2}u) D_{j}u dxdt \nonumber
  \\
  &= (p-1) \int_{\Omega_s} a_{ij} |u|^{p-2} D_iu D_ju x_d^{\theta+1} dxdt + (\theta+1) \int_{\Omega_s} a_{dj} |u|^{p-2}u D_{j}u x_d^{\theta} dxdt \nonumber
  \\
  &= (p-1) \int_{\Omega_s} a_{ij} |u|^{p-2} D_{i}u D_ju x_d^{\theta+1} dxdt + \frac{\theta+1}{p} \int_{\Omega_s} a_{dd} D_d(|u|^p) x_d^{\theta} dxdt \nonumber
  \\
  &= (p-1) \int_{\Omega_s} \frac{a_{ij}+a_{ji}}{2} |u|^{p-2} D_{i}u D_ju x_d^{\theta+1} dxdt \nonumber
  \\
  &\quad + \frac{\theta+1}{p} \int_{\Omega_s} a_{dd} D_d(|u|^p) x_d^{\theta} dxdt =: I+J.
\end{align}
Let us consider a change of variables $y=y(t,x)$ where
\begin{equation*}
  y_d=x_d, \quad y_i = -\int_0^{x_d} \frac{a_{id}+a_{di}}{2a_{dd}}(t,r)dr + x_i, \quad i=1,\dots,d-1.
\end{equation*}
Then $\partial y_i/\partial x_i=1$ for all $i=1,\dots,d$,
\begin{equation*}
  \frac{\partial y_i}{\partial x_d}=-\frac{a_{id}+a_{di}}{2a_{dd}}(t,x_d), \quad i=1,\dots,d-1,
\end{equation*}
and $\partial y_i/\partial x_j=0$ when $i\neq j$ and $j=1,\dots,d-1$.
This implies that $y=y(t,\cdot)$ is a one-to-one Lipschitz map from $\bR_+^d$ to $\bR_+^d$ whose Jacobian is $1$. Then, by letting $v(t,y)=u(t,x)$, for $0<\kappa<p-1$,
\begin{align*}
  I &= (p-1) \int_{\Omega_s} \tilde{a}_{kl}|v|^{p-2} D_{k}v D_lv y_d^{\theta+1} dydt
  \\
  &= (p-1-\kappa) \int_{\Omega_s} \tilde{a}_{kl}|v|^{p-2} D_{k}v D_lv y_d^{\theta+1} dydt + \kappa \int_{\Omega_s} \tilde{a}_{kl}|v|^{p-2} D_{k}v D_lv y_d^{\theta+1} dydt
  \\
  &=:I_1+I_2,
\end{align*}
where
\begin{equation*}
  \tilde{a}_{kl}=\sum_{i,j=1}^d \frac{a_{ij}+a_{ji}}{2} \frac{\partial y_k}{\partial x_i} \frac{\partial y_l}{\partial x_j}.
\end{equation*}
Here, one can show that
\begin{equation*}
  \tilde{a}_{dd}=a_{dd}, \quad \tilde{a}_{dl}=\tilde{a}_{kd}=0, \quad k,l=1,\dots,d-1,
\end{equation*}
and there is $\tilde{\nu}>0$ depending on $\nu$ and $d$ so that
\begin{equation} \label{eq4141711}
  \tilde{a}_{kl}\xi_i \xi_j \geq \tilde{\nu}|\xi|^2.
\end{equation}
Since $a_{dd}$ is constant, by Lemma \ref{lem_Hardy},
\begin{align} \label{eq4141712}
  I_1 &\geq (p-1-\kappa) \int_{\Omega_s} a_{dd} |v|^{p-2} (D_{d}v)^2 y_d^{\theta+1} dydt \nonumber
\\
  &\geq \frac{(p-1-\kappa) \theta^2}{p^2} \int_{\Omega_s} a_{dd} |v|^p y_d^{\theta-1} dydt = \frac{(p-1-\kappa) \theta^2}{p^2} \int_{\Omega_s} a_{dd} |u|^p x_d^{\theta-1} dxdt.
\end{align}
For $I_2$, by \eqref{eq4141711},
\begin{equation} \label{eq4141713}
  I_2 \geq \kappa\tilde{\nu} \int_{\Omega_s} |v|^{p-2} |D_{y}v|^2 y_d^{\theta+1} dydt \geq N\kappa \int_{\Omega_s} |u|^{p-2} |D_{x}u|^2 x_d^{\theta+1} dxdt,
\end{equation}
where $N>0$ is independent of $\kappa$.
Due to \eqref{eq4141712} and \eqref{eq4141713},
\begin{align} \label{eq4141715}
  &I=I_1+I_2 \nonumber
  \\
  &\geq \frac{(p-1-\kappa)\theta^2}{p^2} \int_{\Omega_s} a_{dd} |u|^p x_d^{\theta-1} dxdt + N\kappa \int_{\Omega_s} |u|^{p-2} |D_{x}u|^2 y_d^{\theta+1} dxdt.
\end{align}
Next, we consider the term related to $b_i$. Since $b_i=b_i(t,x_d)$ for $i\neq d$, and $u\in C_c^\infty((-\infty,T]\times \bR^d_+)$,
\begin{align} \label{eq5101400}
  \int_{\Omega_s} b_i |u|^{p-2}u D_{i} u x_d^{\theta+1} dxdt &= \int_{\Omega_s} \frac{b_i}{p} D_i(|u|^p) x_d^{\theta+1} dxdt \nonumber
  \\
  &= \int_{\Omega_s} \frac{b_d}{p} D_d(|u|^p) x_d^{\theta+1} dxdt.
\end{align}
Similarly,
\begin{align} \label{eq5101402}
  &- \int_{\Omega_s} \hat{b}_i u D_{i}(|u|^{p-2}u x_d^{\theta}) dxdt \nonumber
  \\
  &= - (p-1) \int_{\Omega_s} \hat{b}_i  |u|^{p-2}u D_{i}u x_d^{\theta} dxdt - \theta \int_{\Omega_s} \hat{b}_d |u|^p x_d^{\theta-1} dxdt \nonumber
  \\
&= - \frac{p-1}{p} \int_{\Omega_s} \hat{b}_d  D_d(|u|^p) x_d^{\theta} dxdt - \theta \int_{\Omega_s} \hat{b}_d |u|^p x_d^{\theta-1} dxdt.
\end{align}
Thus, if we consider $J$ in \eqref{eq5101401} along with \eqref{eq5101400} and \eqref{eq5101402}, then by \eqref{ratio} and \eqref{eq5100023},
\begin{align} \label{eq5101418}
  &J + \int_{\Omega_s} b_i |u|^{p-2}u D_{i} u x_d^{\theta+1} dxdt - \int_{\Omega_s} \hat{b}_i u D_{i}(|u|^{p-2}u x_d^{\theta}) dxdt \nonumber
  \\
  &= \frac{\theta+1+n_b-(p-1)n_{\hat{b}}}{p} \int_{\Omega_s} a_{dd}  D_d(|u|^p) x_d^{\theta} dxdt - \theta \int_{\Omega_s} \hat{b}_d |u|^p x_d^{\theta-1} dxdt \nonumber
  \\
  &= - \theta \int_{\Omega_s} \hat{b}_d |u|^p x_d^{\theta-1} dxdt \nonumber
  \\
  &= -\frac{\theta((\theta+1)+n_b+n_{\hat{b}})}{p} \int_{\Omega_s} a_{dd} |u|^p x_d^{\theta-1} dxdt.
\end{align}
Here we note that \eqref{eq5100023} is used for both the second and third inequalities above.

Now we deal with the terms on the right-hand side of \eqref{eq4132035}.
By using the condition $p\geq 2$ and Young's inequality, for any $\kappa_1,\kappa_2>0$,
\begin{align} \label{eq4260010}
  -& \int_{\Omega_s} D_i (x_d^{\theta-1} |u|^{p-2}u) F_i dxdt \nonumber
  \\
  &= -(p-1)\int_{\Omega_s} |u|^{p-2} D_iu F_i x_d^{\theta-1} dxdt -\theta \int_{\Omega_s} |u|^{p-2}u F_d x_d^{\theta-2} dxdt \nonumber
  \\
  &\leq \kappa_1 \int_{\Omega_s} |u|^{p-2} |D_xu|^2 x_d^{\theta+1} dxdt + \kappa_2 \int_{\Omega_s} |u|^p x_d^{\theta-1} dxdt \nonumber
  \\
  &\quad + N(\kappa_1,\kappa_2) \int_{\Omega_s} |x_d^{-1}F|^p x_d^{\theta-1} dxdt.
\end{align}
By combining \eqref{eq7151434}, \eqref{eq4132035}-\eqref{eq5101401}, \eqref{eq4141715}, and \eqref{eq5101418}-\eqref{eq4260010},
\begin{align} \label{eq11102104}
  &\frac{1}{p}\int_{\bR_+^d} a_0 |u|^p(s,x) x_d^{\theta-1} dx + N\kappa \int_{\Omega_s} |u|^{p-2} |D_{x}u|^2 x_d^{\theta+2} dxdt \nonumber
  \\
  &\quad + \left( \lambda\nu K^{-1} + n_c - \frac{1+n_b+n_{\hat{b}}}{p}\theta - \frac{1+\kappa}{p^2}\theta^2 \right) \int_{\Omega_s} a_{dd} |u|^p x_d^{\theta-1} dxdt \nonumber
  \\
  &\leq \kappa_1 \int_{\Omega_s} |u|^{p-2} |D_xu|^2 x_d^{\theta+1} dxdt + \kappa_2 \int_{\Omega_s} |u|^p x_d^{\theta-1} dxdt \nonumber
  \\
  &\quad + N(\kappa_1,\kappa_2,n_b,n_{\hat{b}},n_c,p) \int_{\Omega_s} |x_d^{-1}F|^p x_d^{\theta-1} dxdt + N \int_{\Omega_s} |u|^{p-1} |f| x_d^{\theta-1} dxdt.
\end{align}
Since $\alpha$ and $\beta$ are two roots of \eqref{quad_div}, one can take a small $\kappa=\kappa(\theta)>0$ so that
\begin{align*}
  &\lambda\nu K^{-1} + n_c - \frac{1+n_b+n_{\hat{b}}}{p}\theta - \frac{1+\kappa}{p^2}\theta^2
  \\
  &\geq n_c - \frac{1+n_b+n_{\hat{b}}}{p}\theta - \frac{\theta^2}{p^2} -\frac{\kappa}{p^2}\theta^2>0.
\end{align*}
Thus, by taking appropriate $\kappa_1,\kappa_2>0$, we have
\begin{align} \label{eq5221631}
   \int_{\bR_+^d} |u|^p(s,x) x_d^{\theta-1} dx + \int_{\Omega_s} |u|^p x_d^{\theta-1} dxdt \leq N \int_{\Omega_s} \left(|x_d^{-1}F|^p + |u|^{p-1}|f|\right) x_d^{\theta-1} dxdt.
\end{align}
 Now we take the supremum over s on both sides of \eqref{eq5221631}, which yields \eqref{eq4111550}.

For the case when $a_{0}=a_{0}(t)$, test \eqref{zero_div} by $\frac{1}{a_0} |u|^{p-2}ux_d^{\theta-1}$, and repeat the above argument.

Next, we deal with $(i)$.
Let us consider $x_d^\gamma u$ instead of $u$, where $(2-2p)\gamma+\theta+1=0$. By \eqref{eq4131657}, we also have the lower-order coefficients satisfying \eqref{ratio} with $n_b=n_{\hat{b}}=\gamma$ and $n_c=-\gamma(\gamma+1)$. Then, due to \eqref{eq5100023}, we can repeat the proof of $(ii)$. One difference is that, instead of \eqref{eq4141712}, we use
\begin{align*}
  I_1 &\geq (p-1-\kappa) \int_{\Omega_s} a_{dd} |v|^{p-2} (D_{d}v)^2 y_d^{\theta+1} dydt \nonumber
  \\
  &\geq (p-1-\kappa)\nu \int_{\Omega_s} |v|^{p-2} (D_{d}v)^2 y_d^{\theta+1} dydt \nonumber
  \\
  &\geq \frac{(p-1-\kappa)\nu \theta^2}{p^2} \int_{\Omega_s} |v|^p y_d^{\theta-1} dydt \nonumber
\\
  &\geq \frac{(p-1-\kappa)\nu^2 \theta^2}{p^2} \int_{\Omega_s} a_{dd} |v|^p y_d^{\theta-1} dydt = \frac{(p-1-\kappa)\nu^2 \theta^2}{p^2} \int_{\Omega_s} a_{dd} |u|^p x_d^{\theta-1} dxdt,
\end{align*}
which follows from \eqref{ellip} and Lemma \ref{lem_Hardy}.
Then from the same argument above, instead of \eqref{eq11102104}, for any $\kappa,\kappa_1,\kappa_2>0$,
\begin{align*}
  &\frac{1}{p}\int_{\bR_+^d} a_0 |u|^p(s,x) x_d^{\theta-1} dx + N\kappa \int_{\Omega_s} |u|^{p-2} |D_{x}u|^2 x_d^{\theta+2} dxdt
  \\
  &\quad + \left( \lambda\nu K^{-1} + n_c - \frac{1+n_b+n_{\hat{b}}}{p}\theta + \frac{(p-1-\kappa)\nu^2-p}{p^2}\theta^2 \right) \int_{\Omega_s} a_{dd} |u|^p x_d^{\theta-1} dxdt
  \\
  &\leq \kappa_1 \int_{\Omega_s} |u|^{p-2} |D_xu|^2 x_d^{\theta+1} dxdt + \kappa_2 \int_{\Omega_s} |u|^p x_d^{\theta-1} dxdt
  \\
  &\quad + N(\kappa_1,\kappa_2) \int_{\Omega_s} |x_d^{-1}F|^p x_d^{\theta-1} dxdt + N \int_{\Omega_s} |u|^{p-1} |f| x_d^{\theta-1} dxdt.
\end{align*}
Here, we note that $n_b,n_{\hat{b}}$, and $n_c$ may not be zero, and they depend on $\theta$. Thus, we choose appropriate $\kappa,\kappa_1,\kappa_2$, and take $\lambda_0=\lambda_0(d,p,\nu,\theta,K)\geq0$ such that for $\lambda\geq\lambda_0$,
\begin{align*}
  &\lambda\nu K^{-1} + n_c - \frac{1+n_b+n_{\hat{b}}}{p}\theta + \frac{(p-1-\kappa)\nu^2-p}{p^2}\theta^2>0.
\end{align*}
Thus we obtain \eqref{eq5221631}, and the lemma is proved. 
\end{proof}

By repeating the proof of Lemma \ref{lem_zero_div} with $u\in C_c^\infty([0,T]\times \bR^d_+)$, we have the following result.

\begin{lemma} \label{lem_zero_initial}
Let $T\in(0,\infty]$, $\theta\in\bR$, $\lambda\geq0$, and $p\in[2,\infty)$. Assume that $u\in C_c^\infty([0,T]\times\bR^d_+)$ satisfies 
\begin{equation*}
\sL_pu +\lambda c_0 u = D_iF_i+f
\end{equation*}
in $(0,T)\times\bR^d_+$, where $x_d^{-1}F,f\in \bL_{p,\theta}(0,T)$.

$(i)$ There exists a sufficiently large $\lambda_0=\lambda_0(d,p,\theta,\nu,K)\geq0$
such that under Assumption \ref{ass_simple1}, the following holds true. For any $\lambda\geq\lambda_0$, we have
\begin{align} \label{eq11102158}
  &\sup_{t\leq T} \int_{\bR^d_+} |u|^p(t,\cdot) x_d^{\theta-1} dx + \int_{\Omega_T} |u|^p x_d^{\theta-1} dxdt \nonumber
  \\
  &\leq N \int_{\Omega_T} |u(0,\cdot)|^p x_d^{\theta-1} dx + N \int_{\Omega_T} \left(|x_d^{-1}F|^p + |u|^{p-1}|f|\right) x_d^{\theta-1} dxdt,
\end{align}
where $N=N(d,p, \theta,\nu,K)$.

$(ii)$ Let $n_b,n_{\hat{b}},n_c\in\bR$, and assume the quadratic equation \eqref{quad_div} has two distinct real roots $\alpha <\beta$.
Then under Assumption \ref{ass_simple2}, \eqref{eq11102158} holds for any $\lambda\geq0$ and $\alpha p<\theta <\beta p$, where the constant $N$ depends on $d,p,\theta,n_b,n_{\hat{b}},n_c,\nu$, and $K$.
\end{lemma}

\begin{proof}
  One just needs to repeat the proof of Lemma \ref{lem_zero_div} with $u\in C_c^\infty([0,T]\times \bR^d_+)$. We remark that instead of \eqref{eq6281608}, we use
  \begin{align*}
  &\int_0^s\int_{\bR^d_+} a_0u_t |u|^{p-2}u x_d^{\theta-1} dxdt 
  \\
  &= \frac{1}{p}\int_{\bR_+^d} a_0|u|^p(s,x) x_d^{\theta-1} dx - \frac{1}{p}\int_{\bR_+^d} a_0|u|^p(0,x) x_d^{\theta-1} dx
\end{align*}
for any $s\in[0,T]$.
The lemma is proved.
\end{proof}

Next, we obtain higher-order estimates for solutions to equations with general coefficients. We first introduce some function spaces in the whole space $(S,T)\times \bR^d$ for $-\infty\leq S<T\leq \infty$.
For a given weight $\omega(t)$ on $(S,T)$, we write $L_{p,\omega}(S,T):=L_p((S,T)\times\bR^d,\omega(t) dtdx)$. In the case when $S=-\infty,$ we write $L_{p,\omega}(T):=L_{p,\omega}(-\infty,T)$. If $h,D_xh \in L_{p,\omega}(T)$ and $a_0h_t \in L_{p}((S,T),\omega dt; H_p^{-1}(\bR^d))$, then we say $h\in\cH_{p,\omega}^1(S,T)$. We also denote $\cH_{p,\omega}^1(T)=\cH_{p,\omega}^1(-\infty,T)$.

\begin{lemma} \label{lem_high_div}
Let $\rho_0\in(1/2,1)$, $T\in(-\infty,\infty]$, $\lambda\geq0$, $p\in(1,\infty)$, $\theta\in \bR$, $K_0\geq1$, and $[\omega]_{A_p}\leq K_0$. Assume that $u\in C_c^\infty((-\infty,T]\times\bR^d_+)$ satisfies 
\begin{equation*}
\sL_pu +\lambda c_0 u = D_iF_i+f
\end{equation*}
in $\Omega_T$, where $f,x_d^{-1}F\in \bL_{p,\theta,\omega}(T)$.
Then there exists a sufficiently small $\gamma_0>0$ depending only on $d,p,\nu$, and $K_0$ such that under Assumption \ref{ass_lead} $(\rho_0,\gamma_0)$, we have
\begin{align} \label{eq4151444}
  &(1+\sqrt{\lambda}) \|u\|_{\bL_{p,\theta,\omega}(T)} + \|x_dD_xu\|_{\bL_{p,\theta,\omega}(T)} \nonumber
  \\
  &\leq N \left( \frac{1}{1+\sqrt{\lambda}}\|f\|_{\bL_{p,\theta,\omega}(T)} + \|x_d^{-1}F\|_{\bL_{p,\theta,\omega}(T)} + \|u\|_{\bL_{p,\theta,\omega}(T)}\right),
\end{align}
where $N=N(d,p,\theta,\nu,K,K_0)$.
\end{lemma}

\begin{proof}
  For any function $h$ on $\Omega_T$, we denote $h_r(t,x):=h(t,x/r)$. For instance,
\begin{equation*}
  u_r(t,x)=u(t,x/r), \quad a_{ij,r}(t,x)=a_{ij}(t,x/r).
\end{equation*}
Let $\zeta\in C_0^\infty((2,3))$ be a standard non-negative cut-off function.
   Then for any $\gamma\in\bR$ and a function $h$ on $\Omega_T$,
  \begin{align} \label{eq4152140}
    \int_0^\infty \left( \int_{\Omega_{T}} |D_d^j\zeta(x_d)h_r(t,x)|^p \omega(t) dxdt \right) r^\gamma dr &= N_j \int_{\Omega_T} |h(t,x)|^p x_d^{-\gamma-d-1} \omega(t) dxdt,
  \end{align}
  where 
  \begin{equation*}
    N_j:=\int_0^\infty |D_s^j\zeta(s)|^p s^{\gamma+d} ds, \quad j=0,1.
  \end{equation*}

  Now we let $v(t,x):=\zeta(x_d)u_r(t,x) \in \cH_{p,\omega}^{1}(T)$. Then for $\bar{\lambda}\geq0$ which will be specified below, $v$ satisfies
  \begin{equation*}
    a_{0,r}v_t- D_i(x_d^2a_{ij,r}D_{j}v)+ (\lambda+\bar{\lambda})c_{0,r} v = D_iG_i + g
  \end{equation*}
  in $(-\infty,T)\times \bR^d$, where
\begin{align} \label{eq4152142}
  G_i:=r\zeta F_{i,r} - x_d \hat{b}_{i,r}\zeta u_r - x_d^2 a_{id,r}\zeta' u_r
\end{align}
and
  \begin{align} \label{eq4152143}
    g&:=\zeta f_r - r\zeta' F_{d,r} - x_d^2 a_{dj,r}\zeta' D_ju_r - 2x_d a_{dj,r}\zeta D_j  u_r \nonumber
    \\
    &\quad  - x_d b_{j,r} \zeta D_iu_r + x_d\hat{b}_{d,r}\zeta'u + \hat{b}_{d,r} \zeta u_r - c_r \zeta u_r +\bar{\lambda} c_{0,r} \zeta u_r.
  \end{align}
Due to Assumption \ref{ass_lead} $(\rho_0,\gamma_0)$, by scaling, one can show that the oscillation of $x_d^2 a_{ij,r}$ on $Q_{R}(t,x):=(t-R^2,t)\times B_R(x)$ is less than $N\gamma_0$ if $x_d\in(1,4)$ and $R>0$ is sufficiently small (see also Remark \ref{rem5101542}). Thus, there exist $\gamma_0=\gamma_0(d,p,\nu,K,K_0)$ and $\bar{\lambda}_0=\bar{\lambda}_0(d,p,\nu,K,K_0)$ such that if $\bar{\lambda}\geq\bar{\lambda}_0$, then
we can apply the $\cH_{p,\omega}^1(T)$-estimates (see Remark \ref{rem9162000} below) to get
  \begin{align} \label{eq4152321}
    \sqrt{\lambda+\bar{\lambda}} \|v\|_{L_{p,\omega}(\Omega_T)} + \|D_xv\|_{L_{p,\omega}(\Omega_T)} &\leq N \left(\|G\|_{L_{p,\omega}(\Omega_T)} + \frac{1}{\sqrt{\lambda+\bar{\lambda}}} \|g\|_{L_{p,\omega}(\Omega_T)}\right).
  \end{align}
   By \eqref{eq4152142} and \eqref{eq4152143}, using $supp\,\zeta\subset(2,3)$, we have
  \begin{align} \label{eq4152307}
    &\|G\|_{L_{p,\omega}(T)}  \leq N \left( r\|\zeta F_{r}\|_{L_{p,\omega}(T)} + \|\zeta u_r\|_{L_{p,\omega}(T)} + \|\zeta'u_r\|_{L_{p,\omega}(T)}\right)
  \end{align}
  and
  \begin{align} \label{eq4152308}
    \|g\|_{L_{p,\omega}(T)} &\leq N\|\zeta f_r\|_{L_{p,\omega}(T)} + Nr\|\zeta' F_{r}\|_{L_{p,\omega}(T)} \nonumber
    \\
    &\quad + Nr^{-1}\|\zeta'(D_xu)_r\|_{L_{p,\omega}(T)} + Nr^{-1}\|\zeta(D_xu)_r\|_{L_{p,\omega}(T)} \nonumber
    \\
    &\quad + N\|\zeta' u_r\|_{L_{p,\omega}(T)} + N(1+\bar{\lambda})\|\zeta u_r\|_{L_{p,\omega}(T)}.
  \end{align}
  Now we raise both sides of \eqref{eq4152321} to the power of $p$, multiply by $r^{-\theta-d}$, and integrate with respect to $r$ on $(0,\infty)$. Then by \eqref{eq4152140}, \eqref{eq4152307}, and \eqref{eq4152308}, we have
\begin{align} \label{eq11172310}
  &(1+\sqrt{\lambda})\|u\|_{\bL_{p,\theta,\omega}(T)} + \|x_d D_x u\|_{\bL_{p,\theta,\omega}(T)} \nonumber
  \\
  &\leq \frac{N}{1+\sqrt{\lambda}} \|f\|_{\bL_{p,\theta,\omega}(T)} + N \|x_d^{-1}F\|_{\bL_{p,\theta,\omega}(T)} \nonumber
  \\
  &\quad+ N\left(1 + \sqrt{\bar{\lambda}} \right)\|u\|_{\bL_{p,\theta,\omega}(T)} + \frac{N}{\sqrt{\bar{\lambda}}}\|x_dD_xu\|_{\bL_{p,\theta,\omega}(T)}.
\end{align}
Now we choose a sufficiently large $\bar{\lambda}$ such that $N/\sqrt{\bar{\lambda}}<1/2$.
This easily yields the desired estimate \eqref{eq4151444}. 
  The lemma is proved.
\end{proof}

\begin{remark} \label{rem9162000}
  In the proof of Lemma \ref{lem_high_div}, we used the $\cH_{p,\omega}^1(T)$-estimates for the following type of equation
  \begin{equation*}
  a_0u_t- D_i(a_{ij}D_{j}u)+ \lambda c_{0} u = D_iF_i + f,
\end{equation*}
which is slightly different from the equations considered in the literature; we have more general $a_0$ and $c_0$. In this remark, we outline how to obtain the estimate for this generalized equation.

First, we note that the equation can be reduced to the case $a_0=1$. The case $a_0=a_0(t)$ can be easily handled by dividing the equation by $a_0$. For the case $a_0=a_0(x_d)$, we consider a change of variables $y=y(x)$, where $y'=x'$, and
\begin{equation*}
   y_d=y_d(x_d)=\int_0^{x_d} a_0(z_d) dz_d,
 \end{equation*}
 and then divide the equation by $a_0$.
 Under this transformation, the leading coefficients still satisfy the partially BMO assumptions (see \cite[Assumptions A and A\textquotesingle]{DK11CVPDE}). Next, by following the arguments in \cite{DK11SIAM}, we obtain the desired estimates when $\omega=1$, and $(a_{ij})$ and $c_0$ depend on the same single variable, either $x_d$ or $t$.
For general case, one just needs to repeat the arguments in \cite{DK23}. More specifically, by following the proofs of Proposition $4.1$, Lemma $5.1$, and Theorem $2.8$ in the paper, we can derive the mean oscillation estimates of solutions and obtain the desired $\cH_{p,\omega}^1(T)$-result.
\end{remark}

By repeating the proof of Lemma \ref{lem_high_div} using \cite[Theorem 8.2]{DK11}, instead of the $\cH_{p,\omega}^1(T)$-estimates, we obtain the following result for equations on a finite time interval.

\begin{lemma} \label{lem_high_div_fi}
Let $\rho_0\in(1/2,1)$, $T\in(0,\infty]$, $\lambda\geq0$, $p\in(1,\infty)$, and $\theta\in \bR$. Assume that $u\in C_c^\infty([0,T]\times\bR^d_+)$ satisfies $u(0,\cdot)=0$, and 
\begin{equation*}
\sL_pu +\lambda c_0 u = D_iF_i+f
\end{equation*}
in $\Omega_{0,T}$, where $f,x_d^{-1}F\in \bL_{p,\theta}(0,T)$.
Then there exists a sufficiently small $\gamma_0>0$ depending only on $d,p$, and $\nu$ such that under Assumption \ref{ass_lead} $(\rho_0,\gamma_0)$,
\begin{align*}
  &(1+\sqrt{\lambda}) \|u\|_{\bL_{p,\theta}(0,T)} + \|x_dD_xu\|_{\bL_{p,\theta}(0,T)} \nonumber
  \\
  &\leq N \left( \frac{1}{1+\sqrt{\lambda}} \|f\|_{\bL_{p,\theta}(0,T)} + \|x_d^{-1}F\|_{\bL_{p,\theta}(0,T)} + \|u\|_{\bL_{p,\theta}(0,T)}\right),
\end{align*}
where $N=N(d,p,\theta,\nu,K)$.
\end{lemma}

\begin{remark} \label{rem1118}
  Since Assumption \ref{ass_lead} $(\rho_0,\gamma_0)$ holds if Assumption \ref{ass_coeff} $(\rho_0,\gamma_0)$ is satisfied, the assertions in Lemmas \ref{lem_high_div} and \ref{lem_high_div_fi} remain valid under Assumption \ref{ass_coeff} $(\rho_0,\gamma_0)$.
\end{remark}

\mysection{Equations with simple coefficients} \label{sec_simple}

In this section, we present the solvability of the equations with simple coefficients.

We introduce the following (unweighted) function spaces. For $-\infty\leq S<T\leq \infty$ and $\cD \subset \bR^d$, we set
\begin{equation*}
  \cH_p^1((S,T)\times \cD):= \{u:u_t\in \widetilde{\bH}_{p}^{-1}((S,T)\times \cD), D_x^\alpha u\in L_p((S,T)\times \cD), 0\leq |\alpha|\leq 1 \},
\end{equation*}
where
\begin{equation*}
  \widetilde{\bH}_{p}^{-1}((S,T)\times \cD):=\{v: a_0v= D_i G_i + g, \text{ where } G_i,g\in L_p((S,T)\times \cD) \}.
\end{equation*}
We denote by $\mathring{\cH}_p^1((S,T)\times \cD)$ the closure in $\cH_p^1((S,T)\times \cD)$ of functions $u\in C_c^\infty([S,T]\times \cD)$ such that $u(S,\cdot)=0$.

We prove the solvability result in $\cH_{p,\theta}^1(T)=\cH_{p,p,\theta,1}^1(T)$.

\begin{lemma} \label{lem4192205}
  Let $T\in (-\infty,\infty]$, $\theta\in\bR$, $\lambda\geq0$, $p\in(1,\infty)$, and $f,x_d^{-1}F \in \bL_{p,\theta}(T)$. 

$(i)$ There exists $\lambda_0=\lambda_0(d,p,\theta,\nu,K)\geq0$ such that under Assumption \ref{ass_simple1}, the following assertions hold true. For any $\lambda\geq\lambda_0$, there is a unique solution $u\in \cH_{p,\theta}^1(T)$ to 
\begin{equation} \label{eq4261628}
   \sL_pu + \lambda c_0 u= D_iF_i+f
 \end{equation}
 in $\Omega_T$. Moreover, we have
\begin{align} \label{eq4191709}
  &(1+\sqrt{\lambda})\|u\|_{\bL_{p,\theta}(T)} + \|x_dD_xu\|_{\bL_{p,\theta}(T)} \nonumber
  \\
  &\leq N \left( \|x_d^{-1}F\|_{\bL_{p,\theta}(T)} + \frac{1}{1+\sqrt{\lambda}} \|f\|_{\bL_{p,\theta}(T)} \right),
\end{align}
where $N=N(d,p,\theta,\nu,K)$.

$(ii)$ Let $n_b,n_{\hat{b}},n_c\in\bR$, and assume that the quadratic equation 
\begin{equation} \label{eq4212254}
  z^2+(1+n_b+n_{\hat{b}})z-n_c=0
\end{equation}
has two distinct real roots $\alpha<\beta$.
Then under Assumption \ref{ass_simple2}, the assertions in $(i)$ hold with $\lambda_0=0$ and $\alpha p<\theta<\beta p$, where the dependencies of the constant $N$ are replaced by $d,p,\theta,n_b,n_{\hat{b}},n_c,\nu$, and $K$.
\end{lemma}

\begin{proof}
Since the proofs for $(i)$ and $(ii)$ are the same, we present them together.

\textbf{1.} Case $T=\infty$.

Assume for the moment that we obtain the a priori estimate \eqref{eq4191709}.
Due to the method of continuity, it is enough to prove the existence when $a_{ij}, b_i, \hat{b}_i, c$, and $c_0$ are constants, and $f,F\in C_c^\infty((-\infty,T)\times \bR^d_+)$.
Let $A_k:=(T-k,T)\times \bR^{d-1}\times (2^{-k},2^k)$.
Since $f, F\in L_p(A_k)$, by \cite[Theorem 8.2]{DK11}, there is a solution $u_k\in \mathring{\cH}_p^1(A_k)$ of \eqref{eq4261628} with the initial condition $u_k(T-k,\cdot)=0$.
Here, we remark that although $a_0$ is not constant in our case, the same result as in \cite{DK11} can be obtained by following the proof of the theorem.
Since $u_k(T-k,\cdot)=0$, letting $u_k(t,\cdot)=0$ on $(-\infty,T-k]$, $u_k \in \mathring{\cH}_p^1(\tilde{A}_k)$, where $\tilde{A}_k:=(-\infty,T)\times \bR^{d-1}\times (2^{-k},2^k)$.
 
Now we formally take $|u_k|^{p-2} u_k x_d^{\theta-1}$ or $\frac{1}{a_0}|u|^{p-2}ux_d^{\theta-1}$ as a test function to 
 \begin{equation*}
  \sL_pu + \lambda c_0 u_k = D_i F_i + f
\end{equation*}
 in $\tilde{A}_k$. Then by following the proof of Lemma \ref{lem_zero_div}, it can be shown that there exists $\lambda_0\geq0$ such that
 \begin{align*}
      \int_{\tilde{A}_{k}} |u_k|^p x_d^{\theta-1} dxdt &\leq N \int_{\tilde{A}_{k}} |u_k|^{p-1} \left(|D_x F| + |f|\right) x_d^{\theta-1} dxdt,
 \end{align*}
 provided that $\lambda\geq\lambda_0$. Here, to deal with the whole range of $p\in(1,\infty)$, unlike \eqref{eq4260010}, we introduced the norm of $D_x F\in \bL_{p,\theta}(T)$ on the right-hand side instead of that of $x_d^{-1}F$.
Then by Young's inequality,
 \begin{align} \label{eq4261705}
   \int_{\tilde{A}_{k}} |u_k|^p x_d^{\theta-1} dxdt &\leq N \int_{\tilde{A}_{k}} \left(|D_x F|^p + |f|^p\right) x_d^{\theta-1} dxdt \nonumber
   \\
   &\leq N \int_{\Omega_T} \left(|D_x F|^p + |f|^p\right) x_d^{\theta-1} dxdt,
 \end{align}
 provided that $\lambda\geq \lambda_0$.

 Next, we consider $D_xu_k$. Let $\zeta\in C_0^\infty((2,3))$ be a standard non-negative cut-off function. As in the proof of Lemma \ref{lem_high_div}, for each $r>0, k\in\bN$, and $\tilde{A}_{k,r}:=\{(t,x):(t,x/r)\in \tilde{A}_k\}$, $v(t,x):=\zeta(x_d) u_{k,r}(t,x)=\zeta(x_d) u_k(t,x/r) \in \mathring{\cH}_p^1(\tilde{A}_{k,r})$ satisfies the equation
     \begin{equation*} 
     a_{0,r}v_t- D_i(x_d^2a_{ij}D_{j}v)+ (\lambda+\bar{\lambda})c_0 v=D_iG_i + g
    \end{equation*}
    in $\tilde{A}_{k,r}$, where $\bar{\lambda}\geq0$, and $G$ and $g$ are defined as \eqref{eq4152142} and \eqref{eq4152143}. Note that
$supp(v)=supp(\zeta u_{k,r}) \subset \{r2^{-k}\leq x_d\leq r2^k, 2\leq x_d \leq 3\}$, which implies that $supp(v)\subset (-\infty,T]\times\bR^{d-1}\times [a_{r,k},b_{r,k}]$ where $1/2\leq b_{r,k}-a_{r,k}\leq 1$ for each $r\in (2^{1-k},3\times2^{k})$.
   Thus, by applying \cite[Theorem 7.2]{DK18} to $v$ in this region, for $r\in (2^{1-k},3\times2^{k})$ and a sufficiently large $\bar{\lambda}$,
   \begin{eqnarray} \label{eq4261450}
     &r^{-1} \|\zeta D_xu_{k,r}\|_{L_{p}(\bar{A}_{k,r})} \leq N \left(\| D_x(\zeta u_{k,r})\|_{L_{p}(\bar{A}_{k,r})} + \| \zeta' u_{k,r}\|_{L_{p}(\bar{A}_{k,r})}\right)& \nonumber
     \\
     &\leq N\left( \|G\|_{L_{p}(\bar{A}_{k,r})} + \frac{1}{\sqrt{\lambda+\bar{\lambda}}}\|g\|_{L_{p}(\bar{A}_{k,r})} + \| \zeta' u_{k,r}\|_{L_{p}(\bar{A}_{k,r})}\right),&
   \end{eqnarray}
   where the constants $N$ are independent of $k,r$, and $T$. Here, we remark that, although $a_0$ is assumed to be a constant in \cite{DK18}, by following the proof in this paper, the same result can be obtained even for general $a_0$. Moreover, since $v=0$ if $r\notin (2^{1-k},3\times2^{k})$, we still have \eqref{eq4261450} for all $r\in(0,\infty)$.
  As in \eqref{eq4152140}, one can show that for a function $h$ on $\tilde{A}_{k,r}$,
    \begin{align} \label{eq11032036}
    &\int_0^\infty \left( \int_{\tilde{A}_{k,r}} |D_d^j\zeta(x_d)h_r(t,x)|^p dxdt \right) r^\gamma dr \nonumber
    \\
    &= \int_0^\infty \left( \int_{\tilde{A}_{k}} |D_d^j\zeta(rx_d)h(t,x)|^p dxdt \right) r^{\gamma+d} dr \nonumber
    \\
    &= N_j \int_{\tilde{A}_{k}} |h(t,x)|^p x_d^{-\gamma-d-1} dxdt,
\end{align}
  where $h_r(t,x):=h(t,x/r)$ and
  \begin{equation*}
    N_j:=\int_0^\infty |D_s^j\zeta(s)|^p s^{\gamma+d} ds, \quad j=0,1.
  \end{equation*}
  As in \eqref{eq11172310}, by \eqref{eq4261705}, \eqref{eq4261450}, and \eqref{eq11032036},
 \begin{align*} 
   &\int_{\tilde{A}_{k}} |x_dD_xu_k|^p x_d^{\theta-1} dxdt
   \\
   &\leq N \int_{\tilde{A}_{k}} \left(|D_x F|^p + |f|^p + (1+\bar{\lambda})^{p/2} |u_k|^p + \bar{\lambda}^{-p/2} |x_dD_xu_k|^p\right) x_d^{\theta-1} dxdt
      \\
   &\leq N \int_{\tilde{A}_{k}} \left((1+\bar{\lambda})^{p/2}|D_x F|^p + (1+\bar{\lambda})^{p/2}|f|^p + \bar{\lambda}^{-p/2} |x_dD_xu_k|^p\right) x_d^{\theta-1} dxdt.
 \end{align*}
 By taking sufficiently large $\bar{\lambda}$ so that $N\bar{\lambda}^{-p/2} <1/2$, we get
 \begin{align} \label{eq4261706}
   \int_{\tilde{A}_{k}} |x_dD_xu_k|^p x_d^{\theta-1} dxdt \leq N \int_{\tilde{A}_{k}} \left(|D_x F|^p + |f|^p\right) x_d^{\theta-1} dxdt.
 \end{align}

If we take $u_k=0$ in $\Omega_T\setminus\tilde{A}_{k}$, then by \eqref{eq4261705} and \eqref{eq4261706}, $u_k$ is a bounded sequence in $\cH^1_{p,\theta}(T)$. Thus, there is a subsequence still denoted by $u_k$ so that $u_k \rightharpoonup u$ (weakly) in $\cH^1_{p,\theta}(T)$. One can show that by using the weak formulation in Definition \ref{weak}, $u$ is a solution to \eqref{eq4261628}.

Now we prove the a priori estimate \eqref{eq4191709}. Note that Assumption \ref{ass_lead} $(\rho_0,\gamma_0)$ and Assumption \ref{ass_coeff} $(\rho_0,\gamma_0)$ hold when Assumption \ref{ass_simple1} and Assumption \ref{ass_simple2} are satisfied, respectively.
This implies that the case $p\in[2,\infty)$ can be easily obtained by Lemmas \ref{lem_zero_div} and \ref{lem_high_div}, and Remark \ref{rem1118}.
Thus, we only need to prove the estimate when $p\in(1,2)$.

We use a duality argument. We first treat the case when $a_{dd}=a_{dd}(x_d)$. Define the operators
   \begin{equation} \label{eq6201125}
     \cL u:= a_0u_t - x_d^2 D_{i}(a_{ij} D_j u) + x_d b_i D_i u + x_d D_i(\hat{b}_iu) +cu + \lambda c_0 u
   \end{equation}
   and
   \begin{align} \label{eq6201126}
     \cL^* u &:= -a_0u_t - x_d^2 D_{j}(a_{ij} D_i u) -x_d(\hat{b}_i +2a_{id}) D_iu - x_d D_i((b_i+2a_{di})u) \nonumber
     \\
     &\quad +(c-b_d-\hat{b}_d-2a_{dd})u + \lambda c_0 u.
   \end{align}
      Then $\cL^*$ is the dual operator of $\cL$, and the coefficients of $\cL^*$ still satisfy Assumption \ref{ass_simple1} or Assumption \ref{ass_simple2}.
      In particular, under Assumption \ref{ass_simple2}, the corresponding quadratic equation \eqref{eq4212254} for $\cL^*$ is
   \begin{equation*}
     z^2-(3+n_b+n_{\hat{b}})z-(n_c-n_b-n_{\hat{b}}-2)=0,
   \end{equation*}
   which has two real roots $-\beta +1$ and $-\alpha +1$.
 Notice that the dual space of $\bL_{p,\theta}(T)$ is $\bL_{p',\theta'}(T)$ where $p'=p/(p-1)$, and $\theta/p+\theta'/p'=1$.
Let $u\in C_c^\infty(\bR^{d+1}_+)$ and denote
\begin{equation*}
  \cL u=D_iF_i + f,
\end{equation*}
where $x_d^{-1}F, f\in \bL_{p,\theta}(T)$.
 Since $p'>2$, due to the above existence result, there is $\lambda_0\geq0$ such that for $\lambda\geq\lambda_0$, and $g,G=(G_1,\dots,G_d)\in C_c^\infty(\bR^{d+1}_+)$, we can take a solution $v\in \cH_{p',\theta'}^1(T)$ to $\cL^* v = D_iG_i + g$.
 Then,
\begin{align*}
  (u, D_iG_i +g)_{L_2(\bR^{d+1}_+)} &= (u,\cL^* v)_{L_2(\bR^{d+1}_+)} 
  \\
  &= (\cL u,v)_{L_2(\bR^{d+1}_+)} = (D_iF_i+f,v)_{L_2(\bR^{d+1}_+)},
\end{align*}
where $(\cdot,\cdot)_{\bR^{d+1}_+}$ denotes the standard $L_2$ inner product in $\bR^{d+1}_+$.
This together with the corresponding estimate to \eqref{eq4191709} for $v$ yields that
\begin{align*}
  &|(u,D_iG_i +g)_{L_2(\bR^{d+1}_+)}| 
  \\
  &\leq (\|x_d^{-1}F\|_{\bL_{p,\theta}(T)} + \frac{1}{1+\sqrt{\lambda}}\|f\|_{\bL_{p,\theta}(T)}) ((1+\sqrt{\lambda})\|v\|_{\bL_{p',\theta'}(T)} + \|x_dD_xv\|_{\bL_{p',\theta'}(T)})
  \\
  &\leq N(\|x_d^{-1}F\|_{\bL_{p,\theta}(T)} + \frac{1}{1+\sqrt{\lambda}}\|f\|_{\bL_{p,\theta}(T)}) (\|x_d^{-1}G\|_{\bL_{p',\theta'}(T)} + \frac{1}{1+\sqrt{\lambda}}\|g\|_{\bL_{p',\theta'}(T)}).
\end{align*}
Since $G$ and $g$ are arbitrary, we obtain \eqref{eq4191709}. 
For the case when $a_{dd}=a_{dd}(t)$, repeat the above argument with the fact that $\frac{1}{a_0}\cL^*$ is the dual operator of $\frac{1}{a_0}\cL$.

  \textbf{2.} Case $T<\infty$.

We first show the existence of solutions satisfying \eqref{eq4191709}. For this, it suffices to find a solution when $f,F\in C_c^\infty((-\infty,T)\times \bR^d_+)$. In this case, we extend $F$ and $f$ to be zero for $t\geq T$. Then there exists a solution $u\in \cH_{p,\theta}^1(\infty)$ to \eqref{eq4261628}. Moreover, we have
\begin{eqnarray*}
&(1+\sqrt{\lambda})\|u\|_{\bL_{p,\theta}(T)} + \|x_dD_xu\|_{\bL_{p,\theta}(T)} \leq (1+\sqrt{\lambda})\|u\|_{\bL_{p,\theta}(\infty)} + \|x_dD_xu\|_{\bL_{p,\theta}(\infty)}&
  \\
  &\leq N\left(\|x_d^{-1}F\|_{\bL_{p,\theta}(\infty)} + \frac{1}{1+\sqrt{\lambda}}\|f\|_{\bL_{p,\theta}(\infty)} \right)&
  \\
  &= N\left(\|x_d^{-1}F\|_{\bL_{p,\theta}(T)} + \frac{1}{1+\sqrt{\lambda}}\|f\|_{\bL_{p,\theta}(T)} \right).&
\end{eqnarray*}
This easily implies that $u\in \cH_{p,\theta}^1(T)$, and it is a solution to \eqref{eq4261628} satisfying \eqref{eq4191709}.

Lastly, we consider the uniqueness. Assume that $u\in \cH_{p,\theta}^1(T)$ satisfies
\begin{equation} \label{eq6191604}
  \cL u=0,
\end{equation}
where $\cL$ is defined as \eqref{eq6201125}.
 As in \eqref{eq4261705}, if we take $|u|^{p-2}ux_d^{\theta-1}$ or $\frac{1}{a_0}|u|^{p-2}ux_d^{\theta-1}$ as a test function to \eqref{eq6191604}, then we have $u=0$.
The lemma is proved.
\end{proof}

\begin{lemma} \label{lem6291333}
  Let $T\in (0,\infty]$, $\theta\in\bR$, $\lambda\geq0$, $p\in[2,\infty)$, and $f,x_d^{-1}F \in \bL_{p,\theta}(T)$. 

  $(i)$ There exists $\lambda_0=\lambda_0(d,p,\theta,\nu,K)$ such that under Assumption \ref{ass_simple1}, the following assertions hold true. For any $\lambda\geq\lambda_0$, there is a unique solution $u\in \mathring{\cH}_{p,\theta}^1(T)$ to \eqref{eq4261628} in $\Omega_T$. Moreover, we have
\begin{eqnarray*}
  &\sup_{t\in(0,T)}\|u(t,\cdot)\|_{L_{p,\theta}} + (1+\sqrt{\lambda})\|u\|_{\bL_{p,\theta}(0,T)} + \|x_dD_xu\|_{\bL_{p,\theta}(0,T)}&
  \\
  &\leq N \left( \|x_d^{-1}F\|_{\bL_{p,\theta}(0,T)} + \|f\|_{\bL_{p,\theta}(0,T)} \right),&
\end{eqnarray*}
where $N=N(d,p,\theta,\nu,K)$.

$(ii)$ Let $n_b,n_{\hat{b}},n_c\in\bR$, and assume that the quadratic equation \eqref{eq4212254} has two distinct real roots $\alpha<\beta$.
Then under Assumption \ref{ass_simple2}, the assertions in $(i)$ hold with $\lambda_0=0$ and $\alpha p<\theta<\beta p$, where the dependencies of the constant $N$ are replaced by $d,p,\theta,n_b,n_{\hat{b}},n_c,\nu$, and $K$.
\end{lemma}

\begin{proof}
  The claim can be proved by repeating the proof of Lemma \ref{lem4192205}. The differences in the existence part will be described briefly below.
  First, we use \cite{DK11} to obtain a solution $u_k$ to \eqref{eq4261628} in $\mathring{\cH}_p^1((0,T)\times \bR^{d-1}\times(2^{-k},2^k))$ with the initial condition $u_k(0,\cdot)=0$. Next, for zeroth order estimate \eqref{eq4261705}, we use Lemma \ref{lem_zero_initial} instead of Lemma \ref{lem_zero_div}. We also remark that $\sup_{t\in(0,T)}\|u(t,\cdot)\|_{L_{p,\theta}}$ also can be estimated due to $p\geq2$.
   Lastly, for higher-order estimate \eqref{eq4261706}, we again use Lemma \ref{lem_high_div_fi} instead of Lemma \ref{lem_high_div}. The lemma is proved.
\end{proof}

Next, for weighted mixed-norm estimates, we consider a decomposition of the solution.

\begin{lemma} \label{lem5231637}
Let $T\in (-\infty,\infty]$, $p\in(1,\infty)$, and $f, x_d^{-1}F \in \bL_{p,\theta}(T)$. Suppose that $u\in \cH_{p,\theta}^1(T)$ is a solution to \eqref{eq4261628}.

$(i)$ There exists $\lambda_0=\lambda_0(d,p,\theta,\nu,K)\geq0$ such that under Assumption \ref{ass_simple1}, the following assertions hold true.
Let $\lambda\geq\lambda_0$. Then for any $t_0\in(-\infty,T]$ and $r>0$, there exist $v,w\in \cH_{p,\theta}^1(T)$ such that $u=v+w$,
  \begin{align} \label{eq5231452}
    &\left(\aint_{t_0-r}^{t_0} \|w(t)\|_{L_{p,\theta}}^p dt\right)^{1/p} \nonumber
    \\
    &\leq \frac{N}{1+\lambda}\left( \aint_{t_0-r}^{t_0} \|f(t)\|_{L_{p,\theta}}^p dt\right)^{1/p} + \frac{N}{1+\sqrt{\lambda}} \left(\aint_{t_0-r}^{t_0} \|x_d^{-1}F(t)\|_{L_{p,\theta}}^p dt\right)^{1/p},
  \end{align}
  and
  \begin{align} \label{eq5231605}
    \sup_{t\in(t_0-r/2,t_0)} \|v(t)\|_{L_{p,\theta}} \leq N \left(\aint_{t_0-r}^{t_0} \|v(t)\|_{L_{p,\theta}}^p dt\right)^{1/p},
  \end{align}
  where $N=N(d,p,n_b,n_{\hat{b}},n_c,\nu,K)$.

$(ii)$ Let $n_b,n_{\hat{b}},n_c\in\bR$, and assume that the quadratic equation \eqref{eq4212254} has two distinct real roots $\alpha<\beta$.
Then under Assumption \ref{ass_simple2}, the assertions in $(i)$ hold with $\lambda_0=0$ and $\alpha p<\theta<\beta p$, where the dependencies of the constant $N$ are replaced by $d,p,\theta,n_b,n_{\hat{b}},n_c,\nu$, and $K$.
\end{lemma}

\begin{proof}
As in the proof of Lemma \ref{lem4192205}, we prove both $(i)$ and $(ii)$ together.
By a shift of the coordinates, we may assume that $t_0=0$. Take $\lambda_0\geq0$ to be greater than the $\lambda_0$ in Lemma \ref{lem4192205}.
  By Lemma \ref{lem4192205}, for $\lambda\geq\lambda_0$, there is $w\in \cH_{p,\theta}^1(T)$ such that
\begin{align*}
   \sL_p w + \lambda c_0 w = D_iF_i1_{(-r,0)}+f1_{(-r,0)}
 \end{align*}
 and
 \begin{align*}
   \|w\|_{\bL_{p,\theta}(T)}&\leq \frac{1}{1+\sqrt{\lambda}} \left((1+\sqrt{\lambda}) \|w\|_{\bL_{p,\theta}(T)} + \|x_dD_xw\|_{\bL_{p,\theta}(T)} \right)
   \\
   &\leq \frac{N}{1+\sqrt{\lambda}} \left( \frac{1}{1+\sqrt{\lambda}}\|f1_{(-r,0)}\|_{\bL_{q,p,\theta}(T)} + \|x_d^{-1}F1_{(-r,0)}\|_{\bL_{q,p,\theta}(T)} \right).
 \end{align*}
 This easily yields \eqref{eq5231452}.

Let $v:=u-w \in \cH_{p,\theta}^1(T)$. Then $v$ satisfies 
\begin{align} \label{eq5231544}
     \sL_pv + \lambda c_0 v= D_iF_i1_{(-r,0)^c}+f1_{(-r,0)^c}.
\end{align}

Due to the similarity, we only prove \eqref{eq5231605} for the case $a_{dd}=a_{dd}(x_d)$.
 We first consider the case $p\geq2$. Let us formally test \eqref{eq5231544} by $|v|^{p-2}v \zeta x_d^{\theta-1}$, where $\zeta$ is a cut-off function such that $\zeta=1$ on $(-r/2,0]$, $\zeta=0$ on $(-\infty,-r)$, and $|\zeta_{t}|\leq N/r$. Then, we have
   \begin{align} \label{eq5231604}
    &\int_{\Omega_s} a_0 v_t |v|^{p-2}v \zeta x_d^{\theta-1} dxdt + \int_{\Omega_s} a_{ij} D_i(x_d^{\theta+1} |v|^{p-2}v \zeta) D_{j}v dxdt \nonumber
    \\
    &+ \int_{\Omega_s} b_i|v|^{p-2}v \zeta D_{i} v x_d^{\theta} dxdt - \int_{\Omega_T} \hat{b}_i v D_{i}(|v|^{p-2}v x_d^{\theta} \zeta) dxdt \nonumber
    \\
    &+ \int_{\Omega_s} c |v|^p \zeta x_d^{\theta-1} dxdt + \lambda \int_{\Omega_s} c_0 |v|^p \zeta x_d^{\theta-1} dxdt =0.
  \end{align}
  By integration by parts, for $s\in (-r/2,0)$,
\begin{align*}
  &\int_{\Omega_{s}} a_0 v_t |v|^{p-2}v \zeta x_d^{\theta-1} dxdt = \int_{\bR_+^d} \int_{-\infty}^s \frac{a_0}{p} (|v|^p)_t \zeta x_d^{\theta-1} dxdt 
  \\
  &= \frac{1}{p} \int_{\bR_+^d} a_0 |v|^p(s,\cdot) \zeta x_d^{\theta-1} dx - \frac{1}{p} \int_{-\infty}^s \int_{\bR_+^d} a_0 |v|^p \zeta_{t} x_d^{\theta-1} dxdt.
\end{align*}
We control the other terms in \eqref{eq5231604} as in the proof of Lemma \ref{lem_zero_div}. Then one can find $\lambda_0\geq0$ such that for $\lambda\geq\lambda_0$ and $s\in (-r/2,0)$,
\begin{align*}
  \int_{\bR_+^d} |v|^p(s,x) x_d^{\theta-1} \zeta dx &\leq N \int_{-r}^s \int_{\bR_+^d} |v|^p |\zeta_{t}| x_d^{\theta-1} dxdt\leq N r^{-1} \int_{-r}^0 \int_{\bR_+^d} |v|^p x_d^{\theta-1} dxdt.
\end{align*}
Thus, we have \eqref{eq5231605}.

Next, we deal with the case $p<2$ by using a duality argument. Let $\phi\in C_c^\infty(\bR^d_+)$. Then $\partial_t(\phi \zeta)\in \widetilde{\bH}_{p',\theta'}^{-1}(-\infty,0)$, where $p':=p/(p-1)$, $\theta/p+\theta'/p'=1$, and $\cL^*$ is defined as \eqref{eq6201126}. Thus, one can find $g,x_d^{-1}G\in \bL_{p',\theta'}(-\infty,0)$ such that
$$
\cL^*(\phi \zeta)=g+D_iG_i.
$$
Take $\lambda_0\geq0$ to be greater than  both the $\lambda_0$ in Lemma \ref{lem4192205}, and the ones in Lemmas \ref{lem_zero_initial} and \ref{lem6291333} for $p'$ and $\cL^*$ instead of $p$ and $\cL$, respectively. By applying Lemma \ref{lem6291333} in the reverse time direction, for any $s\in(-r/2,0)$ and $\lambda\geq\lambda_0$, there is a solution $\tilde{v}\in \bH_{p',\theta'}^1(-\infty,s)$ to $\cL^*\tilde{v}=g+D_iG_i$ with the zero terminal condition $\tilde{v}(s,\cdot)=0$.
Thus, due to the denseness of the solution space, there are $\tilde{v}_n\in C_c^\infty((-\infty,s]\times\bR^d_+)$, and $g_n,x_d^{-1}G_n\in \bL_{p',\theta'}(-\infty,s)$ such that $\cL^*\tilde{v}_n=g_n+D_i G_{ni}$, $\tilde{v}_n(s,\cdot)=0$, and
\begin{eqnarray*}
  \tilde{v}_n\to \tilde{v} &\text{ in } \, \bH_{p',\theta'}^1(-\infty,s),&
  \\
g_n, x_d^{-1}G_n \to g, x_d^{-1}G &\text{ in } \, \bL_{p',\theta'}(-\infty,s).&
\end{eqnarray*}
Now we apply Lemma \ref{lem_zero_initial} to $\bar{v}_n:=\tilde{v}_n-\phi\zeta$ in $(-r,0)\times\bR^d_+$ with the operator $\cL^*$ (in the reverse time direction). Then for $s\in(-r/2,0)$, we have (note $p'>2$)
\begin{align*}
  &\sup_{t\in(-r,s)} \|\bar{v}_n(t)\|_{L_{p',\theta'}} 
  \\
&\leq N \left( \|g-g_n\|_{\bL_{p',\theta'}(-r,s)} + \|x_d^{-1}(G-G_n)\|_{\bL_{p',\theta'}(-r,s)} + \|\phi\|_{L_{p',\theta'}} \right),
\end{align*}
which yields
\begin{equation} \label{eq6291710}
  \sup_{t\in(-r,s)} \|\bar{v}(t)\|_{L_{p',\theta'}} \leq N\|\phi\|_{L_{p',\theta'}},
\end{equation}
where $\bar{v}:=\tilde{v}-\phi\zeta$.
Let $v_n\in C_c^\infty((-\infty,0]\times\bR^d_+)$ such that $v_n\to v$ in $\cH_{p,\theta}^1(-r,0)$. Then one can also find $f_n,F_n$ such that $\cL v_n=f_n+ D_i F_{ni}$ with $f_n,x_d^{-1}F_n\to0$ in $\bL_{p,\theta}(-r,0)$ as $n\to \infty$. Denote $\cP:=\cL - a_0\partial_t$, and $\cP^*:=\cL^* + a_0\partial_t$. Then $\cP^*$ is the dual operator of $\cP$. Since
\begin{equation*}
  \cL^*\bar{v}_n=g_n+D_i G_{ni} - \cL^*(\phi\zeta)=g_n-g+D_i(G_{ni}-G_i),
\end{equation*}
 by using a duality relation, and the condition $\tilde{v}_n(s,\cdot)=0$, 
\begin{align*}
  -\int_{\bR^d_+} a_0 v_n(s,\cdot) \zeta(s) \phi \, dx &= \int_{\bR^d_+} a_0 v_n(s,\cdot) \bar{v}_n(s,\cdot) dx
  \\
  &=\int_{-r}^s \int_{\bR^d_+} a_0 \left((v_n\zeta)_t \bar{v}_n + v_n\zeta \bar{v}_{nt} \right) dxdt
  \\
  &=\int_{-r}^s \int_{\bR^d_+} \left( a_0(v_n\zeta)_t \bar{v}_n + v_n\zeta \cP^*\bar{v}_n\right) dxdt
    \\
  &\quad+ \int_{-r}^s \int_{\bR^d_+} \left( v_n\zeta (g-g_n) - (G_i-G_{ni})D_i(v_n\zeta) \right) dxdt
  \\
  &=\int_{-r}^s \int_{\bR^d_+} \left( a_0(v_n\zeta)_t \bar{v}_n + \cP(v_n\zeta) \bar{v}_n\right) dxdt
  \\
  &\quad+ \int_{-r}^s \int_{\bR^d_+} \left( v_n\zeta (g-g_n) - (G_i-G_{ni})D_i(v_n\zeta) \right) dxdt
  \\
  &=\int_{-r}^s \int_{\bR^d_+} \left(\left(f_n\zeta + v_n\zeta_t\right) \bar{v}_n - F_{ni}\zeta D_i\bar{v}_n\right) dxdt
  \\
  &\quad+ \int_{-r}^s \int_{\bR^d_+} \left( v_n\zeta (g-g_n) - (G_i-G_{ni})D_i(v_n\zeta) \right) dxdt.
\end{align*}
Letting $n\to\infty$ and applying H\"older's inequality and \eqref{eq6291710},
\begin{align*}
  \left| \int_{\bR^d_+} a_0(s,\cdot)v(s,\cdot) \phi \, dx \right| &\leq \int_{-r}^s \int_{\bR^d_+} |\zeta_t v \bar{v}| dxdt
  \\
   &\leq \|\zeta_t v\|_{\bL_{1,p,\theta}(-r,0)} \sup_{t\in(-r,0)}\|\bar{v}(t,\cdot)\|_{L_{p',\theta'}}
  \\
  &\leq N\|\zeta_t v\|_{\bL_{1,p,\theta}(-r,0)} \|\phi\|_{L_{p',\theta'}}.
\end{align*}
Since $\phi\in C_c^\infty(\bR^d_+)$ is arbitrary, we have
\begin{align*}
  \sup_{t\in(t_0-r/2,t_0)} \|v(t)\|_{L_{p,\theta}} \leq N\|\zeta_t v\|_{\bL_{1,p,\theta}(-r,0)} \leq Nr^{-1}\|v\|_{\bL_{1,p,\theta}(-r,0)},
\end{align*}
which yields \eqref{eq5231605}. The lemma is proved.
\end{proof}

We introduce the maximal function in the time variable. For any functions $h$ defined on $(-\infty,T)$, we denote
\begin{equation*}
  \bM h (t):= \sup_{t\in (s-r,r)} \aint_{s-r}^s |h(r)| dr.
\end{equation*}

Let $\gamma\in(0,1)$, and denote
\begin{equation} \label{eq5291343}
  \cA(s)=\{t<T: \|u(t,\cdot)\|_{L_{p,\theta}} >s \},
\end{equation}
and
\begin{align} \label{eq5291344}
  \cB_\gamma(s) &=\{ t<T: (\bM \|u(t,\cdot)\|_{L_{p,\theta}}^p)^{1/p} \nonumber
  \\
  &\qquad + \frac{\gamma^{-1/p}}{1+\lambda} (\bM \|f(t,\cdot)\|_{L_{p,\theta}}^p)^{1/p} + \frac{\gamma^{-1/p}}{1+\sqrt{\lambda}} (\bM \|x_d^{-1}F(t,\cdot)\|_{L_{p,\theta}}^p)^{1/p} >s\}.
\end{align}
In $\bR$, we write
\begin{equation} \label{eq5241558}
  \cC_R(t):=(t-R,t+R), \quad \widehat{\cC}_R:=\cC_R\cap\{t\leq T\}.
\end{equation}

\begin{lemma} \label{lem5291338}
Let $R>0$, $\gamma\in(0,1)$, and the assumptions of Lemma \ref{lem5231637} $(i)$ be satisfied. Then there exists a sufficiently large constant $\kappa=\kappa(d,p,\theta,\nu,K)>1$ such that for any $t_0\leq T$ and $s>0$, if
\begin{equation} \label{eq5241425}
  |\cC_{R/4}(t_0)\cap \cA(\kappa s)| \geq \gamma |\cC_{R/4}(t_0)|,
\end{equation}
then
\begin{equation*}
  \widehat{\cC}_{R/4}(t_0) \subset \cB_\gamma(s).
\end{equation*}

If the assumptions of Lemma \ref{lem5231637} $(ii)$ are satisfied instead of $(i)$, then the above assertion holds true where the dependencies of $\kappa$ are replaced by $d,p,\theta,n_b,n_{\hat{b}},n_c,\nu$, and $K$.
\end{lemma}

\begin{proof}
By dividing the equation \eqref{eq4261628} by $s$, we may assume that $s=1$. Suppose that there is $r\in \widehat{\cC}_{R/4}(t_0)$ such that
\begin{align} \label{eq5240102}
   &(\bM \|u(t,\cdot)\|_{L_{p,\theta}}^p)^{1/p}(r) \nonumber
   \\
   &\quad+ \frac{\gamma^{-1/p}}{1+\lambda} (\bM \|f(t,\cdot)\|_{L_{p,\theta}}^p)^{1/p}(r) + \frac{\gamma^{-1/p}}{1+\sqrt{\lambda}} (\bM \|x_d^{-1}F(t,\cdot)\|_{L_{p,\theta}}^p)^{1/p}(r) \leq 1.
\end{align}
Let $t_1:=\min\{t_0+R/4,T\}$. Then,
\begin{equation*}
  r\in \widehat{\cC}_{R/4}(t_0)\subset \cC_{R/2}(t_1).
\end{equation*}
By Lemma \ref{lem5231637} and \eqref{eq5240102}, there is $v,w\in \cH_{q,p,\theta}^1(T)$ such that $u=v+w$, and
 \begin{align*}
    &\left(\aint_{\cC_{R/2}(t_1)} \|w(t)\|_{L_{p,\theta}}^p dt\right)^{1/p} \leq N \left(\aint_{\cC_{R}(t_1)} \|w(t)\|_{L_{p,\theta}}^p dt\right)^{1/p}
    \\
    &\leq N \left( \frac{1}{1+\lambda}\aint_{\cC_{R}(t_1)} \|f(t)\|_{L_{p,\theta}}^p dt + \frac{1}{1+\sqrt{\lambda}} \aint_{\cC_{R}(t_1)} \|x_d^{-1}F(t)\|_{L_{p,\theta}}^p dt\right)^{1/p}
    \\
    &\leq \frac{N}{1+\lambda}(\bM \|f(t,\cdot)\|_{L_{p,\theta}}^p)^{1/p}(r) + \frac{N}{1+\sqrt{\lambda}} (\bM \|x_d^{-1}F(t,\cdot)\|_{L_{p,\theta}}^p)^{1/p}(r) \leq N \gamma^{1/p},
  \end{align*}
  and
  \begin{align*}
    &\sup_{t\in \cC_{R/2}(t_1)} \|v(t)\|_{L_{p,\theta}} \leq N \left(\aint_{\cC_{R}(t_1)} \|v(t)\|_{L_{p,\theta}}^p dt\right)^{1/p}
    \\
    &\qquad \leq N \left(\aint_{\cC_{R}(t_1)} \|u(t)\|_{L_{p,\theta}}^p dt\right)^{1/p} + N \left(\aint_{\cC_{R}(t_1)} \|w(t)\|_{L_{p,\theta}}^p dt\right)^{1/p}
    \\
    &\qquad \leq N\bM \|u(t,\cdot)\|_{L_{p,\theta}}^p)^{1/p} + N \left(\aint_{\cC_{R}(t_1)} \|w(t)\|_{L_{p,\theta}}^p dt\right)^{1/p} 
    \\
    &\qquad \leq N(1+\gamma^{1/p}) \leq N=: N_0.
  \end{align*}
  By the triangle inequality and Chebyshev's inequality, for $\kappa>N_0$,
  \begin{align*}
    |\cC_{R/4}(t_0)\cap \cA(\kappa)| &= |\{t\in \widehat{\cC}_{R/4}(t_0); \|u(t,\cdot)\|_{L_{p,\theta}} > \kappa \}|
    \\
    &\leq |\{t\in \cC_{R/2}(t_1); \|u(t,\cdot)\|_{L_{p,\theta}} > \kappa \}|
    \\
    &\leq |\{t\in \cC_{R/2}(t_1); \|v(t,\cdot)\|_{L_{p,\theta}} > N_0 \}|
    \\
    &\quad + |\{t\in \cC_{R/2}(t_1); \|w(t,\cdot)\|_{L_{p,\theta}} > \kappa-N_0 \}|
    \\
    &=|\{t\in \cC_{R/2}(t_1); \|w(t,\cdot)\|_{L_{p,\theta}} > \kappa-N_0 \}|
    \\
    &\leq \int_{\cC_{R/2}(t_1)} (\kappa-N_0)^{-p} \|w(t,\cdot)\|_{L_{p,\theta}}^p dt
    \\
    &\leq N|\cC_{R/2}(t_1)|\gamma(\kappa-N_0)^{-p} \leq N|\cC_{R/4}(t_0)|\gamma(\kappa-N_0)^{-p}.
  \end{align*}
  Thus, by taking a sufficiently large $\kappa>1$ so that $N\gamma(\kappa-N_0)^{-p} < 1/2$, we have
  \begin{equation*}
    |\cC_{R/4}(t_0)\cap \cA(\kappa)|<\gamma |\cC_{R/4}(t_0)|.
  \end{equation*}
  This contradicts \eqref{eq5241425}. The lemma is proved.
\end{proof}

Next, we present the solvability result in weighted mixed-norm spaces when the coefficients satisfy Assumption \ref{ass_simple1} or Assumption \ref{ass_simple2}.

\begin{theorem} \label{thm_simple_div}
    Let $T\in(-\infty,\infty]$, $p,q\in(1,\infty)$ and $K_0$ be a constant such that $[\omega]_{A_q}\leq K_0$.

$(i)$ There exists $\lambda_0=\lambda_0(d,p,q,\theta,\nu,K,K_0)\geq0$ such that under Assumption \ref{ass_simple1}, the following assertions hold true. For any $\lambda\geq\lambda_0$, there is a unique solution $u\in \cH_{q,p,\theta,\omega}^1(T)$ to \eqref{maindiv}. Moreover, for this solution, we have
\begin{align} \label{eq4161346}
  &(1+\sqrt{\lambda})\|u\|_{\bL_{q,p,\theta,\omega}(T)} + \|x_dD_xu\|_{\bL_{q,p,\theta,\omega}(T)} \nonumber
  \\
  &\leq N \left( \|x_d^{-1} F\|_{\bL_{q,p,\theta,\omega}(T)} + \frac{1}{1+\sqrt{\lambda}} \|f\|_{\bL_{q,p,\theta,\omega}(T)} \right),
\end{align}
where $N=N(d,p,q,\theta,\nu,K,K_0,\nu)$.

$(ii)$ Let $n_b,n_{\hat{b}},n_c\in\bR$, and assume that the quadratic equation 
\begin{equation} \label{eq11111138}
  z^2+(1+n_b+n_{\hat{b}})z-n_c=0
\end{equation}
has two distinct real roots $\alpha<\beta$.
Then under Assumption \ref{ass_simple2}, the assertions $(i)$ hold with $\lambda_0=0$ and $\alpha p<\theta<\beta p$, where the dependencies of the constant $N$ are replaced by $d,p,q,\theta,n_b,n_{\hat{b}},n_c,\nu,K$, and $K_0$.
\end{theorem}

\begin{proof}
We present the proofs for $(i)$ and $(ii)$ together.
Let $\lambda_0\geq0$ be greater than $\lambda_0$ in both Lemmas \ref{lem4192205} and \ref{lem5231637}.
Note that if $f,F\in C_c^\infty((-\infty,T)\times\bR^d_+)$, then by Lemma \ref{lem4192205}, there is a solution $u\in \cH_{p,\theta}^1(T)$ to \eqref{maindiv}.
In Steps \textbf{1}-\textbf{4}, we show that this solution $u$ is in the space $\cH_{q,p,\theta,\omega}^1(T)$, and it satisfies \eqref{eq4161346}. Then, due to the denseness of $C_c^\infty((-\infty,T)\times\bR^d_+)$ in $\bL_{q,p,\theta,\omega}(T)$, we obtain the existence of solutions satisfying \eqref{eq4161346}.
Finally, in Step \textbf{5}, we prove the uniqueness result.

\textbf{1.} $q\in(1,\infty)$ and $p\in(1,p_0)$ for some $p_0=p_0(K_0,q)$.

Let $q\in(1,\infty)$ and $p\in(1,p_0)$, where $p_0=p_0(K_0,q)\in(1,q)$ such that $\omega\in A_{q/p}(\bR)$ for any $p\in(1,p_0)$. In this step, we show \eqref{eq4161346}.

Assume for the moment that $u\in \cH_{q,p,\theta,\omega}^1(T)$. 
Let $\lambda\geq \lambda_0$, where $\lambda_0$ is taken from Lemma \ref{lem5231637}. Also, let $\cA$ and $\cB_\gamma$ be defined as \eqref{eq5291343} and \eqref{eq5291344}, respectively. By Lemmas \ref{lem5291338} and \ref{lemcrawl},
\begin{equation*}
  \omega(\cA(\kappa s)) \leq N\gamma^\delta \omega(\cB_\gamma(s)),
\end{equation*}
where $\kappa=\kappa(d,p,n_b,n_{\hat{b}},n_c,\nu,K)>1$, $N=N(K_0)$, and $\delta=\delta(K_0)$.
Since 
\begin{equation*}
  \|u\|_{\bL_{q,p,\theta,\omega}(T)}^q = q\int_{0}^\infty \omega(\cA(s)) s^{q-1} ds = q\kappa^q \int_0^\infty \omega(\cA(\kappa s)) s^{q-1} ds,
\end{equation*}
by the weighted Hardy-Littlewood theorem (see e.g. \cite[Theorem 2.2]{DK18}),
\begin{align*}
  &\|u\|_{\bL_{q,p,\theta,\omega}(T)}^q \leq N\kappa^q \gamma^\delta \int_0^\infty \omega(\cB_\gamma(s)) s^{q-1} ds
  \\
  &\leq N\kappa^q \gamma^\delta \int_0^\infty  \omega\left(\left\{t<T: (\bM \|u(t,\cdot)\|_{L_{p,\theta}}^{p})^{1/p}>s/3 \right\}\right) s^{q-1}ds
  \\
  &\quad+ N\kappa^q \gamma^\delta \int_0^\infty  \omega\left(\left\{t<T: \frac{\gamma^{-1/p}}{1+\lambda} (\bM \|f(t,\cdot)\|_{L_{p,\theta}}^{p})^{1/p} >s/3 \right\}\right) s^{q-1}ds
  \\
  &\quad + N\kappa^q \gamma^\delta \int_0^\infty  \omega\left(\left\{t<T: \frac{\gamma^{-1/p}}{1+\sqrt{\lambda}} (\bM \|x_d^{-1}F(t,\cdot)\|_{L_{p,\theta}}^{p})^{1/p} >s/3 \right\}\right) s^{q-1}ds
  \\
  &\leq N\kappa^q \gamma^\delta \|u\|_{\bL_{q,p,\theta,\omega}(T)}^q 
  \\
  &\quad + N\kappa^q \gamma^{\delta-q/p} \left(\left(\frac{1}{1+\lambda}\right)^q\|f\|_{\bL_{q,p,\theta,\omega}(T)}^q + \left(\frac{1}{1+\sqrt{\lambda}}\right)^q \|x_d^{-1}F\|_{\bL_{q,p,\theta,\omega}(T)}^q \right).
\end{align*}
By taking a sufficiently small $\gamma>0$ such that $N\kappa^q \gamma^\delta<1/2$, we have
\begin{equation*}
  (1+\sqrt{\lambda})\|u\|_{\bL_{q,p,\theta,\omega}(T)} \leq N\left(\frac{1}{1+\sqrt{\lambda}}\|f\|_{\bL_{q,p,\theta,\omega}(T)} + \|x_d^{-1}F\|_{\bL_{q,p,\theta,\omega}(T)}\right).
\end{equation*}
This together with \eqref{eq4151444} yields \eqref{eq4161346} when $u\in \bL_{q,p,\theta,\omega}(T)$ (see also Remark \ref{rem1118} for the case $(ii)$).

Let us define
\begin{equation*}
  \cL_0u:= a_0 u_t- x_d^2\Delta u +x_dn_b D_du + x_d D_d(n_{\hat{b}} u) + n_cu + \lambda c_0u
\end{equation*}
if $a_0=a_0(x_d)$, and
\begin{equation*}
  \cL_0u:= a_0 u_t- x_d^2a_0\Delta u +x_dn_b a_0D_du + x_d D_d(n_{\hat{b}}a_0 u) + n_ca_0u + \lambda c_0u
\end{equation*}
if $a_0=a_0(t)$.
In particular, when $a_0=a_0(t)$, by dividing the equation by $a_0$, we can assume that $a_0=1$.
Suppose that $v\in\cH_{p,\theta}^1(T)$ is a solution to
\begin{equation*}
  \cL_0 v=D_i F_i + f.
\end{equation*}
We claim that $v\in\cH_{q,p,\theta,\omega}^1(T)$.
Let $\eta\in C_c^\infty((0,1))$ be a non-negative function with unit integral, and for any function $h$ defined on $\Omega_T$, denote
   \begin{equation} \label{eq10241430}
     h^{(\varepsilon)}(t,x):=\int_0^\infty h(t-\varepsilon s,x)\eta(s) ds, \quad \varepsilon\in(0,1).
   \end{equation}
Then $v^{(\varepsilon)}\in \cH_{p,\theta}^1(T)\cap L_{\infty}((-\infty,T);H_{p,\theta}^1)$ satisfies $\cL_0 v^{(\varepsilon)} = D_i F^{(\varepsilon)}_i + f^{(\varepsilon)}$.
Since $f,F\in C_c^\infty((-\infty,T)\times\bR^d_+)$, $F^{(\varepsilon)}$ and $f^{(\varepsilon)}$ also have compact supports in $t$. By applying \eqref{eq4191709} with a sufficiently small $T$, we deduce that $v^{(\varepsilon)}$ is compactly supported in $t\in(-\infty,T]$, which easily implies that $v^{(\varepsilon)}\in \cH_{q,p,\theta,\omega}^1(T)$. Thus, we have \eqref{eq4161346} with $v^{(\varepsilon)}$ instead of $u$. Letting $\varepsilon\to0$ in this inequality, the claim is proved.

  Lastly, we deal with equations with general coefficients by using the method of continuity.
For $\kappa\in[0,1]$ and $v\in \cH_{q,p,\theta,\omega}^1(T)\cap \cH_{p,\theta}^1(T)$, there is $w\in \cH_{q,p,\theta,\omega}^1(T)\cap \cH_{p,\theta}^1(T)$ so that
\begin{equation*}
  \cL_0w=\kappa(\cL_0-\cL)v + D_i F_i + f.
\end{equation*}
Moreover, by \eqref{eq4161346},
\begin{align*}
  (1+\sqrt{\lambda})\|w\| + \|x_dD_xw\|\leq N \left( \|x_d^{-1} F\| + \frac{1}{1+\sqrt{\lambda}} \|f\| \right) + N\kappa\left( \|v\| + \|x_dD_xv\| \right),
\end{align*}
where $\|\cdot\|$ is either the $\bL_{p,\theta}(T)$ norm or the $\bL_{q,p,\theta,\omega}(T)$ norm.
Thus, if $N\kappa<1/2$, then one can find a Cauchy sequence $u_n$ in both $\cH_{q,p,\theta,\omega}^1(T)$ and $\cH_{p,\theta}^1(T)$ such that
\begin{equation*}
  \cL_0u_{n+1}=\kappa(\cL_0-\cL)u_n + D_i F_i + f.
\end{equation*}
Letting $n\to \infty$, the common limit, say $\tilde{u}$, is in both $\cH_{q,p,\theta,\omega}^1(T)$ and $\cH_{p,\theta}^1(T)$, and it satisfies
\begin{equation*}
  (1-\kappa)\cL_0 \tilde{u} + \kappa \cL \tilde{u} = D_iF_i + f.
\end{equation*}
Thus, the claim of this step is proved for $(1-\kappa)\cL_0+\kappa\cL$ when $N\kappa<1/2$.
To obtain the desired result for $\kappa=1$,
one just needs to repeat the above argument finitely many times.

\textbf{2.} $q\in(1,\infty)$ and $p\in(1,p_0)$ for some $p_0=p_0(K_0)$.

Take any $q_0\in (1,\infty)$ and consider $p_0=p_0(K_0,q_0)$ introduced in Step \textbf{1}. Then we have \eqref{eq4161346} with $q_0$ and $p\in(1,p_0)$. Using this and the extrapolation theorem in \cite[Theorem 2.5]{DK18}, \eqref{eq4161346} holds for any $q\in(1,\infty)$ and $p\in (1,p_0)$. Thus, $p_0$ in Step \textbf{1} can be chosen independently of $q$.

\textbf{3.} $q\in(1,\infty)$ and $p\in (p_0/(p_0-1),\infty)$.

We use a duality argument to prove \eqref{eq4161346} for $q\in(1,\infty)$ and $p\in (p_0/(p_0-1),\infty)$, where $p_0$ is taken from Step \textbf{2}.

We first treat the case $T=\infty$. 
Let us consider the operator $\cL$ and its dual operator $\cL^*$, which are defined as \eqref{eq6201125} and \eqref{eq6201126}.
 Notice that the dual space of $\bL_{q,p,\theta,\omega}(T)$ is $\bL_{q',p',\theta',\omega'}(T)$ where
\begin{eqnarray*}
  q'=q/(q-1), \quad p'=p/(p-1), \quad \theta/p+\theta'/p'=1, \quad \omega'=\omega^{-1/(p-1)}.
\end{eqnarray*}
Here, $\omega'\in A_{p'}(\bR)$, $p'\in(1,p_0)$, and the condition $\alpha p<\theta <\beta p$ is equivalent to $-\beta p'+p'<\theta'<-\alpha p'+p'$. Let $G,g\in C_c^\infty(\bR^{d+1}_+)$ and find a solution $v\in \cH_{q',p',\theta',\omega'}^1(T)$ to $\cL^* v = D_i G_i + g$. Note that here we need to choose an appropriate $\lambda_0\geq0$.
 Since $u$ is a solution to \eqref{maindiv}, for the case $a_{dd}=a_{dd}(x_d)$,
\begin{align*}
  (u,D_i G_i + g)_{L_2(\bR^{d+1}_+)} = (u,\cL^* v)_{L_2(\bR^{d+1}_+)} = (D_i F_i+f,v)_{L_2(\bR^{d+1}_+)}.
\end{align*}
This together with the corresponding estimate for $v$ yields that
\begin{align*}
  &|(u,D_iG_i+g)_{L_2(\bR^{d+1}_+)}| \nonumber
  \\
  &\leq N(\|x_d^{-1}F\|_{\bL_{q,p,\theta,\omega}(T)}+\frac{1}{1+\sqrt{\lambda}}\|f\|_{\bL_{q,p,\theta,\omega}(T)}) \nonumber
  \\
  &\quad\times ((1+\sqrt{\lambda}) \|v\|_{\bL_{q',p',\theta',\omega'}(T)}+\|x_dD_xv\|_{\bL_{q',p',\theta',\omega'}(T)}) \nonumber
  \\
  &\leq N(\|x_d^{-1}F\|_{\bL_{q,p,\theta,\omega}(T)}+\frac{1}{1+\sqrt{\lambda}}\|f\|_{\bL_{q,p,\theta,\omega}(T)}) \nonumber
  \\
  &\quad \times\left(\|x_d^{-1}G\|_{\bL_{q',p',\theta',\omega'}(T)} + \frac{1}{1+\sqrt{\lambda}}\|g\|_{\bL_{q',p',\theta',\omega'}(T)}\right).
\end{align*}
 Thus, we have \eqref{eq4161346}. We also remark that the case $a_{dd}=a_{dd}(t)$ can be handled similarly.

Next we deal with the case $T<\infty$. Since $F,f\in C_c^\infty((-\infty,T)\times \bR^d_+)$, one can extend $F,f$ to all of $(-\infty,\infty)\times \bR^d_+$ by letting $F=0$ and $f=0$ in $[T,\infty)\times \bR^d_+$. Take a solution $v\in \cH_{p,\theta}^1(\infty)$ to \eqref{maindiv}. Then, by the uniqueness result from Lemma \ref{lem4192205}, we have $u=v$ for $t< T$. Thus, 
\begin{align*}
   \|u\|_{\bH^1_{q,p,\theta,\omega}(T)} \leq \|v\|_{\bH^1_{q,p,\theta,\omega}(\infty)}  &\leq N \left(\|x_d^{-1}F\|_{\bL_{q,p,\theta,\omega}(\infty)} + \frac{1}{1+\sqrt{\lambda}} \|f\|_{\bL_{q,p,\theta,\omega}(\infty)}\right)
  \\
  &= N \left(\|x_d^{-1}F\|_{\bL_{q,p,\theta,\omega}(T)} + \frac{1}{1+\sqrt{\lambda}}\|f\|_{\bL_{q,p,\theta,\omega}(T)}\right),
\end{align*}
which implies \eqref{eq4161346}.

\textbf{4.} $q\in(1,\infty)$ and $p\in[p_0,p_0/(p_0-1)]$.

Finally, we treat the case $p\in[p_0,p_0/(p_0-1)]$. Take $p_1\in(1,p_0)$ and $p_2\in(p_0/(p_0-1),\infty)$, and denote $\theta_1:=\theta p_1/p$ and $\theta_2:=\theta p_2/p$. Then, $\alpha p_i<\theta_i<\beta p_i$ for $i=1,2$.
 Choose $\kappa\in(0,1)$ such that
\begin{equation*}
  \kappa p_1 + (1-\kappa)p_2=p, \quad \kappa \theta_1 + (1-\kappa)\theta_2=\theta.
\end{equation*}
Since $[L_{p_1}(\bR^d),L_{p_2}(\bR^d)]_\kappa=L_{p}(\bR^d)$, one can use the representation \eqref{equivnorm} and the complex interpolation of the spaces (see e.g. \cite[Theorem 2.2.6]{HVVW16}) to get
\begin{equation*}
  [L_{p_1,\theta_1},L_{p_2,\theta_2}]_\kappa=L_{p,\theta},
\end{equation*}
where $[\cdot,\cdot]_\kappa$ denotes the complex interpolation space. Hence, again by the complex interpolation of the spaces,
\begin{align*}
   \bL_{q,p,\theta,\omega}(T) &= L_q((-\infty,T),\omega dt; L_{p,\theta}) = L_q((-\infty,T),\omega dt; [L_{p_1,\theta_1},L_{p_2,\theta_2}]) 
   \\
   &= [L_q((-\infty,T),\omega dt; L_{p_1,\theta_1}), L_q((-\infty,T),\omega dt; L_{p_2,\theta_2})].
 \end{align*}

We claim that if both $u_1\in \cH^1_{p_1,\theta_1}(T)$ and $u_2\in \cH^1_{p_2,\theta_2}(T)$ are solutions of \eqref{maindiv} with $F,f\in C_c^\infty((-\infty,T)\times\bR^d_+)$, then $u_1=u_2$. Indeed, following the proof of \cite[Theorem 4.3.12]{Klec}, $u_k$, introduced in the proof of Lemma \ref{lem4192205}, is in both $\mathring{\cH}_{p_1}^1(A_k)$ and $\mathring{\cH}_{p_2}^1(A_k)$. Thus, the proof of this lemma yields that claim.

By this claim and Steps \textbf{1}-\textbf{3}, $u \in \cH^1_{p_i,\theta_i}(T)$ ($i=1,2$) implies that $u$ is in both $\cH^1_{q,p_1,\theta_1,\omega}(T)$ and $\cH^1_{q,p_2,\theta_2,\omega}(T)$. 
 Thus, the solution operator is well defined from $\bL_{q,p_1,\theta_1,\omega}(T)+\bL_{q,p_2,\theta_2,\omega}(T)$ to $\cH^1_{q,p_1,\theta_1,\omega}(T)+\cH^1_{q,p_2,\theta_2,\omega}(T)$.
Thus, by the complex interpolation of operators (see e.g. \cite[Theorem C.2.6]{HVVW16}), for $p\in[p_0,p_0/(p_0-1)]$, we have the existence of solutions satisfying \eqref{eq4161346}. 

\textbf{5.} Uniqueness.

Let us consider $\cL_0$, which is introduced in Step \textbf{1}.
   Assume that $u\in \cH_{q,p,\theta,\omega}^1(T)$ is a solution to
   \begin{equation*}
     \cL_0u=0.
   \end{equation*}
   We take a cut-off function $\zeta_n$ on $(-\infty,T)$ such that $0\leq \zeta_n\leq 1$, $\zeta_n=1$ on $(T-n,T)$, $\zeta_n=0$ on $(-\infty,T-n-1)$, and $|\zeta_{nt}|\leq N$. Then, $u^{(\varepsilon)}\zeta_n\in \cH_{p,\theta}^1(T)$, and it satisfies
   \begin{equation*}
   \cL_0(u^{(\varepsilon)}\zeta_n)= u^{(\varepsilon)}\zeta_{nt}.
   \end{equation*}
   Here, $u^{(\varepsilon)}$ is defined as \eqref{eq10241430}.
   Since $u^{(\varepsilon)}\zeta_{nt} \in \bL_{p,\theta}(T)\cap \bL_{q,p,\theta,\omega}(T)$, we can apply the above Steps \textbf{1}-\textbf{4} to obtain
   \begin{equation*}
     \|u^{(\varepsilon)}\zeta_n\|_{\bH_{q,p,\theta,\omega}^1(T)} \leq \|u^{(\varepsilon)}\zeta_{nt}\|_{\bL_{q,p,\theta,\omega}(T)}.
   \end{equation*}
   Since the right-hand side converges to $0$ as $n$ goes to $\infty$, $u^{(\varepsilon)}=0$, which easily implies we have $u=0$.

   For equations with more general coefficients, we use the method of continuity.
Assume that for $\kappa\in[0,1]$ and $u\in \cH_{q,p,\theta,\omega}^1(T)$,
   \begin{equation*}
     (1-\kappa)\cL_0 u + \kappa \cL u=0.
   \end{equation*}
   We can rewrite this equation into $\cL_0u=\kappa(\cL_0-\cL)u$. Since we have the uniqueness result for $\cL_0$, one can apply \eqref{eq4161346} to get
   \begin{equation*}
     \|u\|_{\bH_{q,p,\theta,\omega}^1(T)} \leq N\frac{\kappa}{1+\sqrt{\lambda}} \|u\|_{\bH_{q,p,\theta,\omega}^1(T)}.
   \end{equation*}
   Thus, if we choose $\kappa\in(0,1)$ so that $N\kappa<1/2$, then $u=0$, which implies the uniqueness result for $(1-\kappa)\cL_0 + \kappa \cL$. By repeating this argument finitely many times, we obtain the uniqueness result for $\kappa=1$.
The lemma is proved.
\end{proof}

\begin{remark} \label{rem1101}
  As described in Step \textbf{4} of the proof of Theorem \ref{thm_simple_div}, if $f,F\in C_c^\infty(\Omega_T)$, then a solution $u$ to \eqref{maindiv} is independent of $q,p,\theta$, and $\omega$.
\end{remark}

\mysection{Equations with partially mean oscillation coefficients} \label{sec_bmo}

\begin{lemma} \label{lem_supp_div}
  Let $T\in(-\infty,\infty]$, $\rho_0\in(1/2,1)$, $\gamma_0>0$, $p\in(1,\infty)$ and $K_0$ be a constant such that $[\omega]_{A_p}\leq K_0$. Suppose that $u\in \cH_{p,\theta,\omega}^1(T)$ satisfies 
  \begin{equation*}
  \sL_pu +\lambda c_0 u = D_i F_i +f,
  \end{equation*}
  where $x_d^{-1}F,f\in \bL_{p,\theta,\omega}(T)$, and $u, F$, and $f$ are compactly supported on $(-\infty,T]\times B_{\rho_0}(x_{0})$ for some $x_0\in \bR^d_+$ with $x_{0d}=1$.

$(i)$ 
Let $p_1\in(1,p)$. Then there exists $\lambda_0=\lambda_0(d,p,p_1,\theta,\nu,K,K_0)$ such that under Assumption \ref{ass_lead} $(\rho_0,\gamma_0)$, the following assertion holds.
For any $\lambda\geq\lambda_0$, and $\varepsilon>0$,
  \begin{align} \label{eq4222332}
    &(1+\sqrt{\lambda})\|u\|_{\bL_{p,\theta,\omega}(T)} \nonumber
    \\
    &\leq (\varepsilon + N_{\rho_0,\varepsilon} \gamma_0^{(p-p_1)/pp_1}) ((1+\sqrt{\lambda})\|u\|_{\bL_{p,\theta,\omega}(T)} + \|x_dD_xu\|_{\bL_{p,\theta,\omega}(T)}) \nonumber
    \\
    &\quad + N \left(\|x_d^{-1}F\|_{\bL_{p,\theta,\omega}(T)} + \frac{1}{1+\sqrt{\lambda}}\|f\|_{\bL_{p,\theta,\omega}(T)}\right),
  \end{align}
  where $N$ depends only on $d,p,p_1,\theta,\nu,K$, and $K_1$, and $N_{\rho_0,\varepsilon}$ depends only on $d,p,p_1,\theta,\nu,K,K_1$, and $\varepsilon$.

  $(ii)$ Let $n_b,n_{\hat{b}},n_c\in\bR$ and the quadratic equation \eqref{eq11111138} has two distinct real roots $\alpha$ and $\beta$. Let $p_1\in(1,p)$ and $\theta\in(\alpha p,\beta p)$ such that $\alpha p_1<\theta<\beta p_1$.
   Then under Assumption \ref{ass_coeff} $(\rho_0,\gamma_0)$, the assertion in $(i)$ holds with $\lambda_0=0$, where $N$ depends only on $d,p,p_1,\theta,n_b,n_{\hat{b}},n_c,\nu,K,\nu$, and $K_0$, and $N_{\rho_0,\varepsilon}$ depends only on $d,p,p_1,\theta,n_b,n_{\hat{b}},n_c,\nu,K,K_0,\rho_0$, and $\varepsilon$.
\end{lemma}

\begin{proof}
We first assume that the coefficients satisfy one of the following:
\begin{itemize}
  \item $(a_{ij})$ and $c_0$ satisfy Assumption \ref{ass_lead} $(\rho_0,\gamma_0)$, and $(b_i),(\hat{b}_i)$, and $c$ are zero,

  \item $(a_{ij}),(b_i),(\hat{b}_i),c$ and $c_0$ satisfy Assumption \ref{ass_coeff} $(\rho_0,\gamma_0)$.
\end{itemize}
Then we can take the coefficients $[a_{ij}]_{\rho_0,x_0}, [b_{i}]_{\rho_0,x_0}, [\hat{b}_{i}]_{\rho_0,x_0}$, $[c]_{\rho_0,x_0}$, and $[c_0]_{\rho_0,x_0}$ so that these along with $a_0$ satisfy Assumption \ref{ass_simple1} or Assumption \ref{ass_simple2}. Here, when Assumption \ref{ass_lead} $(\rho_0,\gamma_0)$ is satisfied, we choose $[b_{i}]_{\rho_0,x_0}=[\hat{b}_{i}]_{\rho_0,x_0}=[c]_{\rho_0,x_0}=0$.
  Let $\lambda_0\geq0$ be taken from Theorem \ref{thm_simple_div}.
  By Theorem \ref{thm_simple_div}, there exists a solution $v\in \cH_{p,\theta,\omega}^1(T)$ to
  \begin{equation*}
    \cL_0v=D_iF_i+f,
  \end{equation*}
  where
  \begin{align*} 
    \cL_0 v &:= a_0v_t - x_d^2 D_i([a_{ij}]_{\rho_0,x_0} D_{j}v) + x_d [b_{i}]_{\rho_0,x_0}D_iv 
    \\
    &\quad+ x_dD_i([\hat{b}_{i}]_{\rho_0,x_0} v) +[c]_{\rho_0,x_0}v + \lambda [c_0]_{\rho_0,x_0} v.
  \end{align*}
  Here, we also have
  \begin{align} \label{eq4241018}
    (1+\sqrt{\lambda})\|v\|_{\bL_{p,\theta,\omega}(T)} \leq N \left(\|x_d^{-1}F\|_{\bL_{p,\theta,\omega}(T)} + \frac{1}{1+\sqrt{\lambda}}\|f\|_{\bL_{p,\theta,\omega}(T)} \right).
  \end{align}
  Note that $x_d^{-1}F,f\in \bL_{p,p_1,\theta,\omega}(T)$ since $F$ and $f$ are compactly supported on $(-\infty,T]\times B_{\rho_0}(x_{0})$. Thus, by Remark \ref{rem1101}, $v$ is also in the space $\cH_{p,p_1,\theta,\omega}^1(T)$.
  
  Define $w:=u-v$, which satisfies
  \begin{align*}
    \cL_0w &= x_d^2D_i\left((a_{ij}-[a_{ij}]_{\rho_0,x_0})D_ju \right) - x_d (b_i - [b_{i}]_{\rho_0,x_0}) D_iu 
    \\
    &\quad - x_dD_i\left( (\hat{b}_i - [\hat{b}_{i}]_{\rho_0,x_0}) u\right) - (c-[c]_{\rho_0,x_0})u -\lambda (c_0-[c_0]_{\rho_0,x_0})u.
  \end{align*}
  Since $w\in \cH_{p,p_1,\theta,\omega}^1(T)$, by applying \eqref{eq4161346} with $(p,p_1)$ instead of $(q,p)$,
  \begin{align} \label{eq5131514}
  (1+\sqrt{\lambda})\|w\|_{\bL_{p,p_1,\theta,\omega}(T)} &\leq N \|(a_{ij}-[a_{ij}]_{\rho_0,x_0})x_dD_{x}u\|_{\bL_{p,p_1,\theta,\omega}(T)} \nonumber
  \\
  &\quad + N \|(b_{i}-[b_{i}]_{\rho_0,x_0})x_dD_{x}u\|_{\bL_{p,p_1,\theta,\omega}(T)} \nonumber
  \\
  &\quad + N \|(\hat{b}_{i}-[\hat{b}_{i}]_{\rho_0,x_0}) u\|_{\bL_{p,p_1,\theta,\omega}(T)} \nonumber
  \\
  &\quad + N \|(c-[c]_{\rho_0,x_0}) u\|_{\bL_{p,p_1,\theta,\omega}(T)} \nonumber
  \\
  &\quad + N(1+\sqrt{\lambda}) \|(c_0-[c_0]_{\rho_0,x_0}) u\|_{\bL_{p,p_1,\theta,\omega}(T)}.
  \end{align}
  Here, for the last term, we used $\lambda/(1+\sqrt{\lambda}) \leq 1+\sqrt{\lambda}$.
Since $supp \, u\subset (-\infty,T]\times B_{\rho_0}(x_{0d})$ and $1-\rho_0\leq x_d\leq 1+\rho_0$ in $B_{\rho_0}(x_{0})$, by H\"older's inequality, and Assumption \ref{ass_lead} $(\rho_0,\gamma_0)$ or Assumption \ref{ass_coeff} $(\rho_0,\gamma_0)$,
\begin{align*}
  &\|(a_{ij}-[a_{ij}]_{\rho_0,x_0})x_dD_{x}u\|_{\bL_{p,p_1,\theta,\omega}(T)} \nonumber
  \\
  &\leq \left(\sup_{t\leq T}\int_{B_{\rho_0}(x_0)} |a_{ij}-[a_{ij}]_{\rho_0,x_0}|^{p_1q} x_d^{\theta-1} dx\right)^{1/p_1q} \|x_d D_xu\|_{\bL_{p,\theta,\omega}(T)} \nonumber
  \\
  &\leq N \left( \sup_{t\leq T} \aint_{B_{\rho_0}(x_0)} |a_{ij}-[a_{ij}]_{\rho_0,x_0}| dx \right)^{1/p_1q} \|x_d D_xu\|_{\bL_{p,\theta,\omega}(T)} \nonumber
  \\
  &\leq N \gamma_0^{1/p_1q}\|x_dD_xu\|_{\bL_{p,\theta,\omega}(T)},
\end{align*}
where $q:=p/(p-p_1)$ is the H\"older conjugate of $p/p_1$. By considering the last four terms of \eqref{eq5131514} in a similar way, we have
\begin{align} \label{eq9302343}
  (1+\sqrt{\lambda})\|w\|_{\bL_{p,p_1,\theta,\omega}(T)} \leq N \gamma_0^{1/p_1q}((1+\sqrt{\lambda})\|u\|_{\bL_{p,\theta,\omega}(T)} + \|x_dD_xu\|_{\bL_{p,\theta,\omega}(T)}).
\end{align}
 Note that by the (unweighted) Gagliardo-Nirenberg interpolation inequality in $x$-variable,
\begin{align*}
  \|h\|_{L_{p}(\bR^d)} &\leq N\|D_xh\|_{L_{p}(\bR^d)}^\kappa \|h\|_{L_{p_1}(\bR^d)}^{1-\kappa},
\end{align*}
where 
\begin{equation*}
  \frac{1}{p}=\kappa\left(\frac{1}{p}-\frac{1}{d}\right)+(1-\kappa)\frac{1}{p_1}.
\end{equation*}
Thus, by using \eqref{equivnorm}, for $\zeta \in C_c^\infty(\bR_+)$ satisfying \eqref{eq716908},
\begin{align*}
  &\|w\|_{\bL_{p,\theta,\omega}(T)} \nonumber
  \\
  &= \left( \int_{-\infty}^T \sum_{m=-\infty}^\infty e^{m(\theta+d-1)}\|w(t,e^m\cdot)\zeta\|_{L_p(\bR^d)}^p \omega(t) dt \right)^{1/p} \nonumber
  \\
  &\leq N \left( \int_{-\infty}^T \sum_{m=-\infty}^\infty e^{m(\theta+d-1)}\|w(t,e^m\cdot)\zeta\|_{W_p^1(\bR^d)}^{\kappa p} \|w(t,e^m\cdot)\zeta\|_{L_{p_1}(\bR^d)}^{(1-\kappa) p} \omega(t) dt \right)^{1/p} \nonumber
  \\
  &\leq N\|w\|_{\bH^1_{p,\theta,\omega}(T)}^{\kappa}\|w\|_{\bL_{p,p_1,\theta,\omega}(T)}^{1-\kappa}.
\end{align*}
This and \eqref{eq9302343} yield
\begin{align*}
  &(1+\sqrt{\lambda})\|w\|_{\bL_{p,\theta,\omega}(T)} \nonumber
  \\
  &\leq  \varepsilon (1+\sqrt{\lambda})\|w\|_{\bH_{p,\theta,\omega}^1(T)} + N_{\varepsilon} (1+\sqrt{\lambda})\|w\|_{\bL_{p,p_1,\theta,\omega}(T)} \nonumber
  \\
  &\leq  \varepsilon (1+\sqrt{\lambda})\|w\|_{\bH_{p,\theta,\omega}^1(T)} + N_{\rho_0, \varepsilon} \gamma_0^{1/p_1q}((1+\sqrt{\lambda})\|u\|_{\bL_{p,\theta,\omega}(T)} + \|x_dD_xu\|_{\bL_{p,\theta,\omega}(T)}) \nonumber
  \\
  &\leq (\varepsilon + N_{\rho_0,\varepsilon} \gamma_0^{1/p_1q})((1+\sqrt{\lambda})\|u\|_{\bL_{p,\theta,\omega}(T)} + \|x_dD_xu\|_{\bL_{p,\theta,\omega}(T)}) 
  \\
  &\quad+ \varepsilon (1+\sqrt{\lambda}) \|v\|_{\bH_{p,\theta,\omega}^1(T)}.
\end{align*}
Combining this with \eqref{eq4241018}, we have \eqref{eq4222332}.
Thus, the cases $(i)$ when $(b_i),(\hat{b}_i)$, and $c$ are zero, and $(ii)$ are proved.

It remains to prove $(i)$ where $(b_i),(\hat{b}_i)$, and $c$ are non-zero. Since
  \begin{equation*}
  a_0u_t - x_d^2 D_i(a_{ij}D_{j}u) +\lambda c_0 u = D_i \left(F_i - x_d\hat{b}_iu \right) + \left(f - x_d b_iD_iu + \hat{b}_iu -cu\right),
  \end{equation*}
  by \eqref{eq4151444} and \eqref{eq4222332},
  \begin{align*}
    &(1+\sqrt{\lambda})\|u\|_{\bL_{p,\theta,\omega}(T)}
    \\
    &\leq (\varepsilon + N_{\rho_0,\varepsilon} \gamma_0^{(p-p_1)/pp_1}) ((1+\sqrt{\lambda})\|u\|_{\bL_{p,\theta,\omega}(T)} + \|x_dD_xu\|_{\bL_{p,\theta,\omega}(T)}) \nonumber
    \\
    &\quad + N \left(\|x_d^{-1}F\|_{\bL_{p,\theta,\omega}(T)} + \frac{1}{1+\sqrt{\lambda}}\|f\|_{\bL_{p,\theta,\omega}(T)}\right) \nonumber
    \\
    &\quad + N \left(\|u\|_{\bL_{p,\theta,\omega}(T)} + \frac{1}{1+\sqrt{\lambda}}\|x_dD_xu\|_{\bL_{p,\theta,\omega}(T)}\right) \nonumber
    \\
    &\leq (\varepsilon + N_{\rho_0,\varepsilon} \gamma_0^{(p-p_1)/pp_1}) ((1+\sqrt{\lambda})\|u\|_{\bL_{p,\theta,\omega}(T)} + \|x_dD_xu\|_{\bL_{p,\theta,\omega}(T)}) \nonumber
    \\
    &\quad + N \left(\|x_d^{-1}F\|_{\bL_{p,\theta,\omega}(T)} + \frac{1}{1+\sqrt{\lambda}}\|f\|_{\bL_{p,\theta,\omega}(T)} + \|u\|_{\bL_{p,\theta,\omega}(T)}\right).
  \end{align*}
  Thus, one can choose a sufficiently large $\lambda_0\geq0$, (possibly larger than that in Theorem \ref{thm_simple_div}), to obtain \eqref{eq4222332}.
The lemma is proved.
\end{proof}

The following lemma is taken from \cite[Lemma 5.6]{KL13}.

\begin{lemma} \label{lem716}
  Let $\varepsilon_0>0$. Then there exists $\rho_0=\rho_0(\varepsilon_0)\in(1/2,1)$, and non-negative $\eta_k\in C_c^\infty(\bR^{d}_+)$ such that
  \begin{eqnarray} \label{eq4241508}
    \sum_k \eta_k^p\geq1, \quad \sum_k \eta_k \leq N(d), \quad \sum_k\left( x_d|D_x\eta_k|+ x_d^2|D_x^2\eta_k| \right) \leq \varepsilon_0^p,
  \end{eqnarray}
  and for each $k$, there is a point $x_k\in \bR^{d}_+$ such that $supp\,\eta_k \subset B_{\rho_0 x_{kd}}(x_k)$.
\end{lemma}

Now we are ready to prove Theorem \ref{thm_para_div}.

\begin{proof}[Proof of Theorem \ref{thm_para_div}]
For $(i)$ and $(ii)$, due to the method of continuity and Theorem \ref{thm_simple_div}, it suffices to prove the a priori estimate \eqref{estdiv} when $u\in C_c^\infty((-\infty,T]\times\bR^d_+)$.

We first deal with the case $q=p$. Let $\lambda_0\geq0$ be taken from Lemma \ref{lem_supp_div}, and $u\in C_c^\infty((-\infty,T]\times\bR^d_+)$ satisfy
\begin{equation*}
  \sL_p u+\lambda c_0 u= D_iF_i+f,
\end{equation*}
where $x_d^{-1}F,f\in \bL_{p,\theta,\omega}(T)$.

  Let $\varepsilon>0$, which will be specified below. Take $\eta_k\in C_c^\infty(\bR^{d+1}_+)$ satisfying \eqref{eq4241508} with $\varepsilon_0\in(0,1)$. Here, for each $k$, there is a point $(t_k,x_k)\in \bR^{d+1}_+$ such that $supp\,\eta_k \subset B_{\rho_0 x_{kd}}(x_k)$. Then $u_k:=u\eta_k$ satisfies
   \begin{align*}
     \sL_p u_k +\lambda c_0 u_k&= f\eta_k - F_i D_i\eta_k - x_d^2 a_{ij}D_ju D_i\eta_k + 2x_da_{dj}uD_j\eta_k + x_d b_i uD_i\eta_k 
     \\
     &\quad + x_d \hat{b}_i u D_i\eta_k +D_i\left(F_i\eta_k - x_d^2 a_{ij}uD_j\eta_k \right).
   \end{align*}
By applying Lemma \ref{lem_supp_div} to $v(t,x)=u_k(t,x_{kd}x)$ with $p_1\in(1,p)$ such that $\alpha p_1<\theta<\beta p_1$,
\begin{align*}
  &(1+\sqrt{\lambda})\|u_k\|_{\bL_{p,\theta,\omega}(T)} 
  \\
  &\leq (\varepsilon + N_{\rho_0,\varepsilon} \gamma_0^{(p-p_1)/pp_1})((1+\sqrt{\lambda})\|u_k\|_{\bL_{p,\theta,\omega}(T)} + \|x_dD_xu_k\|_{\bL_{p,\theta,\omega}(T)} )
  \\
  &\quad + \frac{N}{1+\sqrt{\lambda}} \|f\eta_k\|_{\bL_{p,\theta,\omega}(T)} + N \|x_d^{-1}F\eta_k\|_{\bL_{p,\theta,\omega}(T)} + \frac{N}{1+\sqrt{\lambda}} \|F D_x\eta_k\|_{\bL_{p,\theta,\omega}(T)}
  \\
  &\quad+ N \|x_d u D_x\eta_k\|_{\bL_{p,\theta,\omega}(T)} + \frac{N}{1+\sqrt{\lambda}} \|x_d^2 D_xu D_x\eta_k\|_{\bL_{p,\theta,\omega}(T)}.
\end{align*}
By raising both sides of this inequality to the power of $p$, and summing in $k$, \eqref{eq4241508} yields that
\begin{align*}
  &\|u\|_{\bL_{p,\theta,\omega}(T)} \leq (1+\sqrt{\lambda})\|u\|_{\bL_{p,\theta,\omega}(T)} 
  \\
  &\leq (N\varepsilon + N_{\rho_0,\varepsilon} \gamma_0^{(p-p_1)/pp_1}) ((1+\sqrt{\lambda})\|u\|_{\bL_{p,\theta,\omega}(T)} + \|x_dD_xu\|_{\bL_{p,\theta,\omega}(T)})
  \\
  &\quad +  \frac{N}{1+\sqrt{\lambda}} \|f\|_{\bL_{p,\theta,\omega}(T)} + N \|x_d^{-1}F\|_{\bL_{p,\theta,\omega}(T)} 
  \\
  &\quad + N\varepsilon_0\left( \| u\|_{\bL_{p,\theta,\omega}(T)}+\| x_d D_x u\|_{\bL_{p,\theta,\omega}(T)} \right).
\end{align*}
This together with higher-order estimate \eqref{eq4151444} yields (see also Remark \ref{rem1118})
\begin{align*}
  &(1+\sqrt{\lambda}) \|u\|_{\bL_{p,\theta,\omega}(T)} + \|x_dD_xu\|_{\bL_{p,\theta,\omega}(T)} 
  \\
  &\leq \frac{N}{1+\sqrt{\lambda}} \|f\|_{\bL_{p,\theta,\omega}(T)} + N \|x_d^{-1}F\|_{\bL_{p,\theta,\omega}(T)}
  \\
  &\quad + (N\varepsilon + N_{\rho_0,\varepsilon} \gamma_0^{(p-p_1)/pp_1}) ((1+\sqrt{\lambda})\|u\|_{\bL_{p,\theta,\omega}(T)} + \|x_dD_xu\|_{\bL_{p,\theta,\omega}(T)})
  \\
  &\quad + N\varepsilon_0\left( \| u\|_{\bL_{p,\theta,\omega}(T)}+\| x_d D_x u\|_{\bL_{p,\theta,\omega}(T)} \right).
\end{align*}
Then we first choose $\varepsilon_0$ sufficiently small such that $N\varepsilon_0 <1/3$. Then $\rho_0=\rho_0(\varepsilon_0)$ is determined from Lemma \ref{lem716}. Next we take $\varepsilon$ small enough, and then choose $\gamma_0$ sufficiently small so that $N\varepsilon + N_{\varepsilon} \gamma_0^{(p-p_1)/pp_1}<1/3$. Then we obtain \eqref{estdiv}, which proves the case $p=q$.

Now we treat the case $q\neq p$. Let $u\in C_c^\infty((-\infty,T]\times \bR^d_+)$ satisfy \eqref{maindiv}.
From the above case $q=p$, for any $\omega'\in A_p(\bR)$, if $x_d^{-1}F,f\in \bL_{p,\theta,\omega'}(T)$, then $u$ satisfies \eqref{estdiv} with $(p,\omega')$ instead of $(q,\omega)$. Using this and the extrapolation theorem (see e.g. \cite[Theorem 2.5]{DK18}), we have $u\in \cH_{q,p,\theta,\omega}^1(T)$ satisfying the a priori estimate \eqref{estdiv}. The theorem is proved.
\end{proof}

We finish this section by giving the proof of Theorem \ref{thm_finite_div}.

\begin{proof}[Proof of Theorem \ref{thm_finite_div}]
  
  We follow the idea of the proof of \cite[Theorem 2.1]{KryVMO}.

We first prove the uniqueness by showing the a priori estimate \eqref{eq10021627}. Let $u\in C_c^\infty([0,T]\times\bR^d)$ be a solution to \eqref{div_finite} with $u(0,\cdot)=0$. Due to zero initial condition, if we extend $u,F$, and $f$ to be zero for $t\leq0$, then $u\in \cH_{q,p,\theta,\omega}^1(T)$ and it satisfies \eqref{div_finite} in $(-\infty,T)\times\bR^d_+$.
Let $\bar{\lambda}\geq0$, which will be specified below.
Then we see that $v:=e^{-\bar{\lambda}t}u \in \cH_{q,p,\theta,\omega}^1(T)$ satisfies 
  \begin{equation*}
   \sL_p v +(\lambda c_0+\bar{\lambda}a_0)v =D_i(e^{-\bar{\lambda}t}F_i)+e^{-\bar{\lambda}}tf
 \end{equation*}
 in $\Omega_T$.
 Note that $(a_{ij})$, and $\frac{\lambda c_0+\bar{\lambda}a_0}{\lambda+\bar{\lambda}}$ with $[\frac{\lambda c_0+\bar{\lambda}a_0}{\lambda+\bar{\lambda}}]_{\rho,x_0}=\frac{\lambda [c_0]_{\rho,x_0}+\bar{\lambda}a_0}{\lambda+\bar{\lambda}}$ satisfy Assumption \ref{ass_coeff} $(\rho_0,\gamma_0)$. Thus, by Theorem \ref{thm_para_div} $(i)$, there exist $\rho_0\in(1/2,1)$ sufficiently close to $1$, a sufficiently small number $\gamma_0>0$,
 and a sufficiently large number $\lambda_0\geq0$ such that for any $\bar{\lambda}\geq\lambda_0$, we have
  \begin{align*}
  &(1+\sqrt{\lambda+\bar{\lambda}}) \|v\|_{\bL_{q,p,\theta,\omega}(T)} + \|x_dD_xv\|_{\bL_{q,p,\theta,\omega}(T)}
  \\
  &\leq N\left( \|x_d^{-1}e^{-\bar{\lambda}t}F\|_{\bL_{q,p,\theta,\omega}(T)} + \frac{1}{1+\sqrt{\lambda+\bar{\lambda}}}\|e^{-\bar{\lambda}t}f\|_{\bL_{q,p,\theta,\omega}(T)}\right),
 \end{align*}
 where $N$ is independent of $T$.
  By taking sufficiently large $\bar{\lambda}>0$ so that $\frac{N}{1+\sqrt{\lambda+\bar{\lambda}}}<\frac{1}{2}$, we obtain \eqref{eq10021627} with $(v, e^{-\bar{\lambda}t}F,e^{-\bar{\lambda}t}f)$ in place of $(u,F,f)$.
Thus,
 \begin{align*}
     &(1+\sqrt{\lambda}) \|u\|_{\bL_{q,p,\theta,\omega}(0,T)} + \|x_dD_xu\|_{\bL_{q,p,\theta,\omega}(0,T)}
     \\
     &\leq N(T) \left((1+\sqrt{\lambda+\bar{\lambda}}) \|v\|_{\bL_{q,p,\theta,\omega}(T)} + \|x_dD_xv\|_{\bL_{q,p,\theta,\omega}(T)}\right)
  \\
  &\leq N(T) \left( \|x_d^{-1}e^{-\bar{\lambda}t}F\|_{\bL_{q,p,\theta,\omega}(T)} + \frac{1}{1+\sqrt{\lambda+\bar{\lambda}}}\|e^{-\bar{\lambda}t}f\|_{\bL_{q,p,\theta,\omega}(T)}\right)
  \\
  &\leq N(T) \left( \|x_d^{-1}F\|_{\bL_{q,p,\theta,\omega}(T)} + \frac{1}{1+\sqrt{\lambda+\bar{\lambda}}}\|f\|_{\bL_{q,p,\theta,\omega}(T)}\right),
 \end{align*}
 which proves \eqref{eq10021627}.

 Next, we consider the existence. 
Due to the a priori estimate \eqref{eq10021627}, we only need to prove the existence for a given $\lambda\geq0$ when $F,f\in C_c^\infty((0,T)\times\bR^d_+)$, and $b_i=\hat{b}_i=c=0$. Let us extend $f$ and $F$ by zero for $t\leq0$. By Theorem \ref{thm_para_div}, one can find $\rho_0\in(1/2,1)$, $\gamma_0>0$, and $\lambda_0\geq0$ so that under Assumption \ref{ass_lead} $(\rho_0,\gamma_0)$, for any $\lambda\geq\lambda_0$, there is a solution $u\in \cH_{q,p,\theta,\omega}^1(T)$ to
\begin{equation*}
  a_0 u_t - x_d^2 D_i(a_{ij}D_{j}u) +\lambda c_0 u = D_iF + f
\end{equation*}
 in $(-\infty,T)\times\bR^d_+$.
Since both $f$ and $F$ have compact supports in $t\in(0,T)$, by applying \eqref{estdiv} with $T=\delta$, $u(t,\cdot)=0$ for $t<\delta$ for some $\delta>0$. 
Thus, one can deduce that there is $u_n\in C_c^\infty((-\infty,T]\times\bR^d_+)$ such that $u_n(0,\cdot)=0$, and $u_n\to u$ in $\cH_{q,p,\theta,\omega}^1(T)$. This implies that $u\in \mathring{\cH}_{q,p,\theta,\omega}^1(T)$, and it is a solution to \eqref{div_finite} in $(0,T)\times\bR^d_+$. The theorem is proved.
\end{proof}

\mysection{Elliptic equations} \label{sec_ell}

In this section, we deal with the elliptic equations.
The proof of Theorem \ref{thm_ell} is divided into two main parts: $(i)$-$(ii)$ and $(iii)$. For $(i)$-$(ii)$, we will use Theorem \ref{thm_para_div} and repeat the argument presented in Lemma \ref{lem4192205}. To prove $(iii)$, which deals with the special case when $d=1$, we first show the following solvability result.
In this case, we denote $a,b$, and $\hat{b}$, instead of $(a_{ij}), (b_i)$, and $(\hat{b}_i)$, respectively.

\begin{lemma} \label{lem1010953}
Let $p\in(1,\infty)$. Suppose that $a$ is constant, and the coefficients $a, b, \hat{b}$, and $c$ satisfy the ratio condition
      \begin{equation} \label{eq1022135}
    \frac{b}{a}=n_b, \quad \frac{\hat{b}}{a}=n_{\hat{b}}, \quad \frac{c}{a}=n_c
  \end{equation}
  for some $n_b,n_{\hat{b}}, n_c\in \bR$.
  Assume that the quadratic equation 
\begin{equation*}
 z^2+(1+n_b+n_{\hat{b}})z-n_c=0
\end{equation*}
has two distinct real roots $\alpha <\beta$.
Then for any $\theta\in \bR\setminus [\alpha p, \beta p]$, and $x^{-1}F,f\in L_{p,\theta}$, there is a unique solution $u\in H_{p,\theta}^1$ to
 \begin{equation} \label{eq10101330}
-x^2 D_{x}(aD_xu) + x bD_xu + xD_x(\hat{b}u) +cu = D_xF+ f.
\end{equation}
Moreover, for this solution, we have
\begin{eqnarray} \label{eq10101454}
  \|u\|_{L_{p,\theta}} + \|x D_x u\|_{L_{p,\theta}} \leq N \left(\|x^{-1}F\|_{L_{p,\theta}} + \|f\|_{L_{p,\theta}} \right),
\end{eqnarray}
where $N=N(p,\theta,n_b,n_{\hat{b}},n_c,K)$.
\end{lemma}

\begin{proof}
  Since the coefficients satisfy \eqref{ratio}, $b, \hat{b}$, and $c$ are constants. Moreover, we can assume that $\hat{b}=0$.
  Let $0<r<s$ such that $supp(u)\subset(r,s)$. Take a cut-off function $\eta \in C_c^\infty(\bR_+)$ so that $\eta=1$ on $(r,s)$, $\eta=0$ on $(0,r/2)\cup (2s,\infty)$,  and $|xD_x\eta|\leq N$. Then, $v:=u\eta$ satisfies
  \begin{equation*}
    -x^2 aD_x^2v + x bD_xv + xD_x(\hat{b}v) +cv = D_x(F\eta) +\left(f\eta- FD_x\eta\right).
  \end{equation*}
  Since $|FD_x\eta|\leq N|x^{-1}F|$, we can also assume that both $F$ and $f$ have compact supports in $\bR_+$.

  Notice that \eqref{eq10101330} is an ordinary differential equation, which implies that the general solution of this equation is given by
\begin{equation} \label{eq10101418}
  u(x)=(A_1(x)+B_1)x^{-\alpha} + (A_2(x)+B_2)x^{-\beta},
\end{equation}
where $B_1$ and $B_2$ are arbitrary constants,
\begin{align*}
  A_1(x) &:= -\frac{1}{a(\beta-\alpha)} \int_0^x y^{\alpha-1} (D_yF+f) dy 
  \\
  &= -\frac{1}{a(\beta-\alpha)} \int_0^x \left(-\frac{y^{\alpha-2}}{\alpha-1} F + y^{\alpha-1}f\right) dy,
\end{align*}
and
\begin{align*}
  A_2(x) &:= \frac{1}{a(\beta-\alpha)} \int_0^x y^{\beta-1} (D_yF+f) dy 
  \\
  &= \frac{1}{a(\beta-\alpha)} \int_0^x \left(-\frac{y^{\beta-2}}{\beta-1} F + y^{\beta-1}f\right) dy
\end{align*}
(see e.g. \cite[Theorem 3.6.1]{BDM21}). 

Let us consider the range $\theta<\alpha p$. In this case, if we put $B_1=B_2=0$ in \eqref{eq10101418}, then by Hardy’s inequality (see e.g. Theorem 5.1 in the preface of \cite{K85}) 
\begin{align*}
  \|u\|_{L_{p,\theta}} \leq \|A_1\|_{L_{p,\theta-\alpha p}} + \|A_2\|_{L_{p,\theta-\beta p}} \leq N \left(\|x^{-1}F\|_{L_{p,\theta}} + \|f\|_{L_{p,\theta}} \right).
\end{align*}
Thus, we have a solution satisfying \eqref{eq10101454}.

For the case $\theta>\beta p$, we put
\begin{align*}
  B_1&=\frac{1}{a(\beta-\alpha)} \int_0^\infty \left(-\frac{y^{\alpha-2}}{\alpha-1} F + y^{\alpha-1}f\right) dy,
  \\
  B_2&=\frac{1}{a(\beta-\alpha)} \int_0^\infty \left(\frac{y^{\beta-2}}{\beta-1} F - y^{\beta-1}f\right) dy.
\end{align*}
Then again by Hardy's inequality, we obtain \eqref{eq10101454}.

For the uniqueness result, we assume that there is a solution $u\in H_{p,\theta}^1$ to \eqref{eq10101330} with $F=f=0$. Since $a$ is constant, $\hat{b}$ is also constant. Thus, by absorbing $x D_x(\hat{b}u)$ to $xbD_xu$, and dividing \eqref{eq10101330} by $a$, we may assume that $\hat{b}=0$, and $a=1$. Moreover, as in \eqref{eq4131657}, by considering $x^\gamma u$ instead of $u$, we can also assume that $c=0$.
Then we have
\begin{equation*}
  x^2D_x(D_xu)-bxD_xu=0,
\end{equation*}
which is equivalent to
\begin{equation*}
  D_x(x^{-b}D_xu)=0.
\end{equation*}
Thus, we deduce that
\begin{equation*}
  u(x)=C_1x^{-\alpha} + C_2x^{-\beta}
\end{equation*}
for some $C_1,C_2\in\bR$. Since $u\in H_{p,\theta}^1$, we have $C_1=C_2=0$, which yields $u=0$.
The lemma is proved.
\end{proof}

\begin{proof}[Proof of Theorem \ref{thm_ell}]

$(i)$ and $(ii)$. First, we prove the a priori estimate \eqref{eq10081100} for the cases $(i)$ and $(ii)$, by following proof of \cite[Theorem 2.6]{KryVMO}. Let $\eta \in C_c^\infty(\bR)$ and $u\in C_c^\infty(\bR^d_+)$. Then $v(t,x):=\eta_n(t)u(x):=\eta(t/n)u(x)$ satisfies 
    \begin{equation*}
      \sL_p v+\lambda c_0 v =D_i(\eta_n F_i) + \eta_n f + \eta_n' u
    \end{equation*}
    in $\bR\times \bR^d_+$.
    Note that for $g\in L_{p,\theta}$,
\begin{align*}
  \|\eta_n g\|_{\bL_{p,\theta}(\infty)}^p=nN_1\|g\|_{L_{p,\theta}}^p, \quad  \|\eta_n' g\|_{\bL_{p,\theta}(\infty)}^p=n^{1-p}N_2\|g\|_{L_{p,\theta}}^p,
\end{align*}
where
\begin{equation*}
  N_1:=\int_0^\infty |\eta|^p \,dt, \quad N_2:=\int_0^\infty |\eta'|^p \,dt.
\end{equation*}
Thus, if we apply \eqref{estdiv} with the case $p=q$ and $\omega=1$, then for $\lambda\geq\lambda_0$,
\begin{align} \label{eq10091455}
  &(1+\sqrt{\lambda}) \|u\|_{L_{p,\theta}} + \|x_dD_xu\|_{L_{p,\theta}} \nonumber
  \\
  &\leq N \left( \|x_d^{-1}F\|_{L_{p,\theta}} + \frac{1}{1+\sqrt{\lambda}}\|f\|_{L_{p,\theta}} + \frac{n^{-1}}{1+\sqrt{\lambda}}\|u\|_{L_{p,\theta}} \right),
\end{align}
where $\lambda_0$ is taken from Theorem \ref{thm_para_div}.
Letting $n\to\infty$, \eqref{eq10081100} is obtained.

Next, we prove the existence. Due to the method of continuity, we may assume that the coefficients are constants. By \cite[Theorem 8.6]{DK11}, there is $\Lambda\geq0$ such that for any $\lambda\geq\Lambda$, we can find a solution $u_k\in W_p^1(B_k)$ to 
\begin{align*}
  \sL_eu +\lambda c_0 u=D_iF_i + f
\end{align*}
in $B_k:=\bR^{d-1}\times (2^{-k},2^k)$. Here, as mentioned in the proof of Lemma \ref{lem4192205}, the result from \cite{DK11} still holds true for a measurable coefficient $K^{-1}\leq c_0\leq K$.  For $\eta \in C_c^\infty(\bR)$, we denote $v_{k,n}(t,x):=\eta_n(t)u_k(x):=\eta(t/n)u_k(x)$. Then $v_{k,n} \in \mathring{\cH}_p^1(\tilde{A}_k)$, where $\tilde{A}_k:=\bR\times \bR^{d-1}\times (2^{-k},2^k)$. Moreover, it satisfies 
\begin{align*}
 \sL_p v_{k,n} +\lambda c_0 v_{k,n}=D_i(\eta_nF_i) + \eta_nf + \eta_n' u_k
\end{align*}
in $\tilde{A}_k$. As in the proof of Lemma \ref{lem4192205}, by following the proof of Lemma \ref{lem_zero_div}, and using \cite[Theorem 7.2]{DK18}, if we extend $v_{k,n}$ to be zero in $(\bR\times \bR^d_+)\setminus \tilde{A}_k$, then we have \eqref{eq10091455} with $v_{k,n}$ in place of $u$.
By letting $n\to\infty$, we conclude that $u_k$ is a bounded sequence in $H_{p,\theta}^1$. Hence, there is a subsequence still denoted by $u_k$ so that $u_k \rightharpoonup u$ (weakly) in $H_{p,\theta}^1$, which proves the desired result when $\lambda\geq \max\{\Lambda,\lambda_0\}$. Again by the method of continuity, we actually obtain the existence for $\lambda\geq\lambda_0$ excluding the condition $\lambda\geq\Lambda$.

$(iii)$ Next, we consider the case $(iii)$. Due to Lemma \ref{lem1010953}, we have the desired result when the coefficients $a, b, \hat{b}$, and $c$ satisfy the ratio condition \eqref{eq1022135}, and $a$ is constant. To deal with general coefficients, one just needs to repeat the proofs of Lemma \ref{lem_supp_div} and Theorem \ref{thm_para_div}, using the corresponding results for elliptic equations instead of parabolic ones.
The proof is completed.
\end{proof}

\appendix

\mysection{A ``crawling of ink spots'' lemma} \label{sec_ink}

We first present some properties of Muckenhoupt weights. Recall the definitions of $\cC_R(t)$ and $\widehat{\cC}_R(t)$ in \eqref{eq5241558}.

\begin{prop}
Let $p\in(1,\infty)$, and $\omega\in A_p(\bR)$ such that $[\omega]_{A_p(\bR)} \leq K_0$.

(i) The measure $\omega(x) dx$ is a doubling measure; for any $t\in \bR$ and $R>1$,
\begin{equation*} 
  \omega(\cC_R(t)) \leq R^p [\omega]_{A_p(\bR)}\omega(\cC_1(t)).
\end{equation*}

(ii) Let $E\subset C_R(t)$ for some $t\in\bR$ and $R>0$. Then there exist $N=N(K_0)>0$ and $\delta=\delta(K_0)\in(0,1)$ such that
\begin{equation*} \label{eq5272053}
  N^{-1}\left(\frac{|E|}{|\cC_R(t)|}\right)^p \leq \frac{\omega(E)}{\omega(\cC_R(t))} \leq N\left(\frac{|E|}{|\cC_R(t)|}\right)^{\delta}.
\end{equation*}
\end{prop}

\begin{proof}
See \cite[Propositions 7.1.5, 7.2.8]{G14}.
\end{proof}

\begin{lemma} \label{lemcrawl}
Let $\gamma \in (0,1)$ and $E \subset F \subset (-\infty,T)$. Suppose that $|E| < \infty$, and for any $t \in (-\infty,T]$ and $R \in (0,\infty)$ with
\begin{equation*}
  \left| \cC_R(t) \cap E \right| \geq \gamma |\cC_R(t)|,
\end{equation*}
we have
\begin{equation*}
  \widehat{\cC}_R(t) \subset F.
\end{equation*}
Then we have
\begin{equation*} \label{eq5272055}
\omega(E) \leq N \gamma^{\delta} \omega(F),
\end{equation*}
where $\delta>0$ and $N>0$ are constants depending only on $K_0$.
\end{lemma}

\begin{proof}
Let $t\in E$, and denote
\begin{equation*}
\varphi_{t}(r) := \frac{|E \cap \cC_r(t)|}{|\cC_r(t)|}.  
\end{equation*}
Then, $\varphi_t(r) \leq \frac{|E|}{|\cC_r(t)|} \to 0$ as $r \to \infty$.
On the other hand, by the Lebesgue differentiation theorem, there is a null set $N_0$ such that
\begin{equation*}
\lim_{r \to 0} \varphi_{t}(r) = 1
\end{equation*}
for any $t\in \widetilde{E}:= E\setminus N_0$.
Since $\gamma \in (0,1)$, and $\varphi_{t}(r)$ is continuous on $(0,\infty)$, for any $t \in \widetilde{E}$, there is $r \in (0,\infty)$ such that
\begin{equation*}
  \varphi_{t}(r) = \gamma.
\end{equation*}
Since $|E| < \infty$, if we define
\begin{equation*}
  R(t) := \sup\{ r \in (0,\infty): \varphi_{t}(r) = \gamma\}, \quad t\in \widetilde{E},
\end{equation*}
then $R(t)$ is uniformly bounded.
We set
\begin{equation*}
  \Gamma_1 := \{\cC_{R(t)}(t): t \in \widetilde{E}, \,\, R(t) < \infty\}.
\end{equation*}
and
\begin{equation} \label{eq5272233}
  R^*_1 := \sup \{R(t): \cC_{R(t)}(t) \in \Gamma_1\}.
\end{equation}
Then $R^*_1<\infty$ and
\begin{equation} \label{eq5281427}
  \widetilde{E} \subset \bigcup_{\cC_{R(t)}(t) \in \Gamma_1} \cC_{R(t)}(t).
\end{equation}

Now we choose a countable sub-collection $\Gamma_0$ of $\Gamma_1$ as follows.
Using \eqref{eq5272233}, we can take $\cC_{R_1}(t_1):=\cC_{R(t_1)}(t_1)$ from $\Gamma_1$ such that $R_1 > R^*_1/2$. Let $\Gamma_2$ be a sub-collection of $\Gamma_1$ whose elements are disjoint from $\cC_{R_1}(t_1)$. Then we denote $\Gamma_2'=\Gamma_1\setminus \Gamma_2$. Here, we note that for $\cC_{R(t)}(t) \in \Gamma_2'$, 
\begin{equation*}
  \cC_{R(t)}(t) \cap \cC_{R_1}(t_1) \neq \emptyset,
\end{equation*}
and
\begin{equation} \label{eq5281418}
  \cC_{R(t)}(t) \subset \cC_{5R_1}(t_1).
\end{equation}

Now we describe the process to choose $\Gamma_k$ for $k\geq3$. Assume that $\cC_{R_k}(t_k)$ and $\Gamma_{k+1}$ are chosen.
If $\Gamma_{k+1}$ is empty, the process ends.
If not, we take $\cC_{R_{k+1}}(t_{k+1})\in \Gamma_{k+1}$ such that $R_{k+1} > \frac{1}{2} R^*_{k+1}$, where
\begin{equation*}
  R^*_{k+1}:= \sup_{\cC_{R(t)}(t) \in \Gamma_{k+1}} R(t).
\end{equation*}
Then, we split $\Gamma_{k+1} = \Gamma_{k+2} \cup \Gamma_{k+2}'$, where $\Gamma_{k+2}$ consists of $\cC_{R(t)}(t)\in \Gamma_{k+1}$ such that $\cC_{R(t)}(t)\cap \cC_{R_{k+1}}(t_{k+1})=\emptyset$, and $\Gamma_{k+2}'=\Gamma_{k+1}\setminus\Gamma_{k+2}$.
Now we define
\begin{equation*}
  \Gamma_0 := \{\cC_{R_k}(t_k) : k \in \bN \}.
\end{equation*}

We claim that there are only two cases where $\Gamma_0$ contains only finitely many elements or has infinitely many elements with $R_k^* \downarrow 0$. Assume that there exists a number $\varepsilon_0 > 0$ such that $R_k^* \geq \varepsilon_0$ for all $k\in\bN$, which leads to $\sum_{k=1}^\infty |\cC_{R_k}(t_k)|=\infty$. This and $|\cC_{R_k}(t_k)\cap E|=\gamma |\cC_{R_k}(t_k)|$ imply that $|E|=\infty$, which is a contradiction. Thus, the claim is proved.

Our next goal is to prove
\begin{equation} \label{eq5272332}
\Gamma_1 = \bigcup_{k=2}^\infty \Gamma_k'.
\end{equation}
Since the case when $\Gamma_0$ contains only finitely many elements is obvious, we only consider the case when $R_k^* \downarrow 0$. Let us take $\cC_{R(t)}(t) \in \Gamma_1$ such that
\begin{equation*}
  \cC_{R(t)}(t) \notin \bigcup_{k=2}^\infty \Gamma_k'.
\end{equation*}
Then $\cC_{R(t)}(t) \in \Gamma_k$ for all $k\in\bN$. However, due to the definition of $R_k^*$, we have $R(t) = 0$, which contradicts $R(t) > 0$. Thus, \eqref{eq5272332} also holds when $R_k^* \downarrow 0$.

Note that as in \eqref{eq5281418},
\begin{equation*}
  \cC_{R(t)}(t) \subset \cC_{5 R_k(t_k)}
\end{equation*}
for any $\cC_{R(t)}(t) \in \Gamma_{k+1}'$ with $k\in\bN$. This, \eqref{eq5281427} and \eqref{eq5272332} yield that
\begin{equation*}
\widetilde{E} \subset \bigcup_{\cC_{R(t)}(t) \in \Gamma_1} \cC_{R(t)}(t) \subset \bigcup_{k=1}^\infty \cC_{5R_k}(t_k).
\end{equation*}
Note that due to $\cC_{R_k}(t_k)\in \Gamma_1$, for $k\in\bN$,
\begin{equation*}
  |\cC_{R_k}(t_k)\cap E|=\gamma|\cC_{R_k}(t_k)|, \quad |\cC_{5R_k}(t_k)\cap E|<\gamma|\cC_{5R_k}(t_k)|,
\end{equation*}
which imply that 
\begin{equation*}
  \widehat{\cC}_{R_k}(t_k) \subset F,
\end{equation*}
and\begin{equation*}
  \omega(\cC_{5R_k}(t_k)\cap E) \leq \gamma^\delta \omega(\cC_{5R_k}(t_k)),
\end{equation*}
where $\delta=\delta(K_0)$ is taken from \eqref{eq5272053}. Since $\cC_{R_k}(t_k)$ are disjoint,
\begin{align*}
\omega(E) &= \omega(\widetilde{E}) \leq \omega\left( \bigcup_{k=1}^\infty \widetilde{E} \cap \cC_{5R_k}(t_k) \right) \leq \sum_{k=1}^\infty \omega(\widetilde{E} \cap \cC_{5R_k}(t_k))
\\
&\leq \gamma^\delta \sum_{k=1}^\infty \omega(\cC_{5R_k}(t_k)) \leq 5^p[\omega]_{A_p(\bR)} \gamma^\delta \sum_{k=1}^\infty \omega(\cC_{R_k}(t_k))  
\\
&= 5^p [\omega]_{A_p(\bR)} \gamma^\delta \omega\left(\bigcup_{k=1}^\infty \cC_{R_k}(t_k)\right) \leq 10^p [\omega]_{A_p(\bR)}^2 \gamma^\delta\omega \left( \bigcup_{k=1}^\infty \widehat{\cC}_{R_k}(t_k)\right) 
\\
&\leq N(K_0) \gamma^\delta \omega(F).
\end{align*}
Hence, we have \eqref{eq5272055}.
The lemma is proved.
\end{proof}





\end{document}